\documentclass[11pt]{amsart}

\usepackage[all]{xy}
\usepackage{amscd,amsmath,amsthm,amssymb}
\usepackage{mathtools}
\usepackage{mathrsfs}
\usepackage{caption}
\usepackage{tikz-cd}
\usepackage{tikz,pgfplots}
\usetikzlibrary{cd}
\usetikzlibrary{shapes}
\usetikzlibrary{babel}
\usetikzlibrary{math}

\newcommand\markedvertex[1]{node[circle,inner sep=0mm, minimum size=1mm, fill=white,draw]{}}

\tikzset{vertex/.style={circle,inner sep=0mm, minimum size=1mm, fill=black, prefix after command= {\pgfextra{\tikzset{every label/.style={font=\scriptsize}}}}}}
\tikzset{mvertex/.style={circle,inner sep=0mm, minimum size=1mm, fill=white, draw, prefix after command= {\pgfextra{\tikzset{every label/.style={font=\scriptsize}}}}}}

\definecolor{cadmiumgreen}{rgb}{0.0, 0.42, 0.24}
\usepackage[
colorlinks, citecolor=cadmiumgreen,
pagebackref,
pdfauthor={Robin de Jong, Stefan van der Lugt}, 
pdfstartview ={FitV},
]{hyperref}
\hypersetup{
pdftitle={Rings of tautological forms on moduli spaces of curves},
}

\usepackage[
msc-links,
nobysame,
lite,
]{amsrefs} 

\usepackage[normalem]{ulem}

\usepackage[left=3.5cm,top=3.5cm,right=3.5cm]{geometry}
\usepackage{enumitem} 

\pgfplotsset{compat=1.13}

\newtheorem{thm}{Theorem}[section]
\newtheorem{prop}[thm]{Proposition}
\newtheorem{cor}[thm]{Corollary}
\newtheorem{lem}[thm]{Lemma}
\newtheorem{lemma}[thm]{Lemma}

\theoremstyle{remark}
\newtheorem{remark}[thm]{Remark}
\newtheorem{example}[thm]{Example}

\theoremstyle{definition}
\newtheorem{definition}[thm]{Definition}

\def\i{\mathrm{i}}
\def\cG{\mathcal{G}}
\def\cM{\mathcal{M}}
\def\cJ{\mathcal{J}}
\def\tR{\mathcal{R}}

\def\cB{\mathcal{B}}
\def\cP{\mathcal{P}}
\def\cX{\mathcal{X}}
\def\cC{\mathcal{C}}

\def\cT{\mathcal{T}}
\def\vv{\mathrm{v}}

\def\cO{\mathcal{O}}
\def\cS{\mathcal{S}}
\def\ee{\mathrm{e}}

\renewcommand{\Im}{\operatorname{Im}}

\newcommand{\R}{\mathbb{R}}
\newcommand{\C}{\mathbb{C}}
\newcommand{\Z}{\mathbb{Z}}
\newcommand{\Q}{\mathbb{Q}}

\newcommand{\inp}[1]{\ensuremath{{\left\langle #1 \right\rangle}}}
\newcommand{\abs}[1]{\ensuremath{{\left\vert #1 \right\vert}}}

\newcommand{\divisor}{\operatorname{div}}

\DeclareMathOperator{\id}{id}

\newcommand{\imto}{\xrightarrow{\sim}}

\newcommand{\pair}[1]{\langle#1\rangle}

\newcommand{\Pic}{\operatorname{Pic}}
\newcommand{\CH}{CH}
\newcommand{\Sp}{\operatorname{Sp}}

\newcommand{\del}{\partial}
\newcommand{\delbar}{\overline{\del}}

\newcommand{\p}{\mathfrak{p}}

\newcommand\xto\xrightarrow
\newcommand\bb\mathbb
\newcommand{\cat}[1]{\mathbf{#1}}

\newcommand{\isoto}{\xrightarrow{\sim}}

\numberwithin{equation}{section}

\subjclass[2010]{
\href{https://mathscinet.ams.org/msc/msc2020.html?t=32G15}{32G15},
\href{https://mathscinet.ams.org/msc/msc2020.html?t=14D23}{14D23},
\href{https://mathscinet.ams.org/msc/msc2020.html?t=14H10}{14H10},
\href{https://mathscinet.ams.org/msc/msc2020.html?t=14H15}{14H15}}

\date{\today}

\begin{document}

\title[Rings of tautological forms on moduli spaces of curves]{Rings of tautological forms on moduli spaces of curves}

\author{Robin de Jong}
\email{\href{mailto:rdejong@math.leidenuniv.nl}{rdejong@math.leidenuniv.nl}}

\author{Stefan van der Lugt}
\email{\href{mailto:math@svdlugt.nl}{math@svdlugt.nl}}

\begin{abstract} We define and study a natural system of tautological rings on the moduli spaces of marked curves at the level of differential forms. We show that certain $2$-forms obtained from the natural normal functions on these moduli spaces are tautological. Also we show that rings of  tautological forms are always  finite dimensional. Finally we characterize the Kawazumi--Zhang invariant as essentially the only smooth function on the moduli space of curves whose Levi form is a tautological form.
\end{abstract}

\dedicatory{To the memory of Bas Edixhoven}
\maketitle

\setcounter{tocdepth}{1}


\thispagestyle{empty}

\section{Introduction} 

In the last four decades, tautological rings on the moduli spaces of curves have been intensively studied, starting with a fundamental paper by Mumford \cite{Mumford1983}. One usually defines these rings as subrings of Chow rings, or of cohomology rings. The aim of the present work is to initiate a study of natural tautological rings at the level of \emph{differential forms}.

\subsection{Tautological rings in Chow and cohomology}

We start by recalling briefly the usual tautological rings. 
Let \(g\geq 2\) be an integer, and let $\cM_g$ be the moduli space of complex connected genus \(g\) curves. 
Let \(p: \cC_g \to \cM_g\) be the universal family. We are interested in $\cM_g$ and the fiber products \(\cC_g^r\) for $r \in \Z_{> 0}$ classifying genus \(g\) curves with $r$ (not necessarily distinct) marked points on them. To each pair of integers \(r,s\geq 0\) and each map of sets \(\phi: \{1,\dots,s\}\to \{1,\dots,r\}\) we have associated a \emph{tautological morphism} 
$ f^\phi: \cC_g^r \to \cC_g^s $
given by sending a marked curve \( (C,x_1,\dots,x_r) \) to the marked curve \( (C,x_{\phi(1)},\dots,x_{\phi(s)}) \). 

In \cite{LooijengaTautological} Looijenga initiated the study of the  \emph{tautological rings} $R^*(\cC_g^r)$ of the moduli spaces of $r$-marked curves. The rings $R^*(\cC_g^r)$ are characterized as the smallest $\Q$-sub-algebras of the Chow rings $\CH^*(\cC_g^r)$ with rational coefficients that are closed under pushforward and pullback along tautological morphisms. At level $r=0$ the system of tautological rings $R^*(\cC_g^r)$ specializes to give the tautological ring $R^*(\cM_g)$ introduced by Mumford \cite{Mumford1983}. Among other things, Looijenga proves in  \cite{LooijengaTautological}  that the tautological rings $R^d(\cC_g^r)$ vanish in degree $d > g-r+2$.

As a variant on the above, one can consider the tautological rings $RH^*(\cC_g^r)$ in \emph{cohomology} with rational coefficients. These rings are obtained as the image of the canonical ring map $ R^*(\cC_g^r) \to H^{2*}(\cC_g^r,\Q)$. Our goal in this paper is to propose a lift of the $\R$-algebras $RH^*(\cC_g^r) \otimes_\Q \R \subset H^*(\cC_g^r,\R)$ to the level of smooth differential forms.  

 \subsection{The rings of tautological forms} \label{def:tautring}

Unfortunately,  the characterization of tautological rings as given above does not  immediately generalize to the setting of differential forms. One issue is that pushforwards of differential forms are  in general not defined. They are defined though for \emph{submersions}, using the process of  \emph{integrating along the fiber}. We therefore decide to restrict pushforwards to the cases where the tautological morphism $f^\phi$ is submersive, i.e., when the map $\phi$ is \emph{injective}. The price we pay for this restriction is that we will give ourselves one specific $2$-form $h$ representing the diagonal class on $\cC_g^2$ as a starting point, and declare it to be tautological. Our choice of the $2$-form $h$ is motivated by the work of Kawazumi \cites{kawazumi2008johnson, kawazumi2009canonical}. 

Let $\Delta$ denote the diagonal class on $\cC_g^2$. We also use $\Delta$ to denote the diagonal morphism $\cC_g \to \cC_g^2$, i.e., the tautological morphism corresponding to the unique map $\{1,2\} \to \{1\}$. The identity $\Delta_*(1)=\Delta$ of classes shows that $\Delta$ is a tautological class. Let $G \colon \cC_g^2 \to \R$ be the canonical Green's function as introduced by Arakelov \cite{Arakelov1974}, see Section~\ref{sec:arakmetric}. The function $G$ defines a smooth Hermitian metric $\|\cdot\|$ on the line bundle $\cO(\Delta)$ on $\cC_g^2$ by setting $\|1\|=G$, where $1$ denotes the canonical global section of the line bundle $\cO(\Delta)$. 

We take the $2$-form $h$ to be the Chern form of the line bundle $\cO(\Delta)$ with the given metric,
\[ h := c_1\left( \cO(\Delta), \|\cdot\| \right). \]
This leads to the following definition for the tautological rings at the level of forms. When \(f:\cC_g^r \to \cC_g^s\) is a tautological submersion we denote by $\int_f \colon A^*(\cC_g^r) \to A^*(\cC_g^s)$ the integration along the fiber operating on differential forms.  \\

\noindent \textbf{Definition.} 
	The \emph{rings of tautological forms} \(\tR^*(\cC_g^r)\) (\(r\geq 0\)) are the unique sub-\(\R\)-algebras \(\tR^*(\cC_g^r) \subseteq A^*(\cC_g^r)\) such that the following holds:
	\begin{enumerate}
		\item \(h\in \tR^*(\cC_g^2)\);
		\item If \(f:\cC_g^r \to \cC_g^s\) is a tautological morphism, then \(f^*(\tR^*(\cC_g^s)) \subseteq \tR^*(\cC_g^r)\);
		\item If \(f:\cC_g^r \to \cC_g^s\) is a tautological submersion, then \(\int_f(\tR^*(\cC_g^r)) \subseteq \tR^*(\cC_g^s)\);
		\item \(\tR^*(\cC_g^r)\) are minimal: if \(S^*(\cC_g^r)\subseteq A^*(\cC_g^r)\) (\(r\geq 0\)) is a collection of sub-\(\R\)-algebras that satisfies (1)--(3), then \(\tR^*(\cC_g^r) \subseteq S^*(\cC_g^r)\) for all \(r\geq 0\).
	\end{enumerate}
Elements of the rings \(\tR^*(\cC_g^r)\) are called \emph{tautological (differential) forms}. \\

As any system satisfying (1)--(3) above can always be reduced to a smaller system satisfying (1)--(3) by removing all odd-degree forms, it follows
that all odd-degree subspaces \(\tR^{2d+1}(\cC_g^r)\) vanish. A similar argument shows that the rings of tautological forms consist entirely of \emph{closed} forms.

From the definition we readily obtain the tautological differential forms
\begin{equation} \label{eqn:e^A}
 e^A := \Delta^*h  \in \tR^2(\cC_g) 
\end{equation}
 as well as
\begin{equation} \label{eqn:e_d^A}
 e_d^A := \int_{\cC_g/\cM_g} (e^A)^{d+1} \in \tR^{2d}(\cM_g) \, , \quad d \in \Z_{> 0} . 
\end{equation}
The notation $e^A$ is borrowed from \cites{kawazumi2008johnson,  kawazumi2009canonical}.

Let $r \in \Z_{\ge 0}$. Based on the examples \eqref{eqn:e^A} and \eqref{eqn:e_d^A}, it is easy to see that (as desired) the canonical map
\( \tR^*(\cC_g^r) \to H^{2*}(\cC_g^r;\R) \) obtained by taking cohomology classes surjects onto the sub-$\R$-algebra $RH^*(\cC_g^r) \otimes_\Q \R$. Indeed, let $p \colon \cC_g^r \to \cM_g$ denote the projection map, for $i=1,\ldots,r$ denote by $p_i \colon \cC_g^r \to \cC_g$ the projection onto the $i$-th coordinate, and for $1 \leq i < j \leq r$ denote by $p_{ij} \colon \cC_g^r \to \cC_g^2$ the projection onto the $i$-th and $j$-th coordinate. Let $K$ denote the cohomology class of the relative cotangent bundle of $\cC_g$ over $\cM_g$, and let $\kappa_d = p_* K^{d+1}$ denote the kappa-classes on $\cM_g$. The surjectivity claim follows immediately from the following observations: 
\begin{itemize}
\item the tautological ring $RH^*(\cC_g^r)$ is generated by the classes $p^*\kappa_d$,  $p_i^*K$ and $p_{ij}^*\Delta$ -- this follows immediately from how the tautological rings are defined in Chow rings in \cite{LooijengaTautological}; 
\item the tautological form $e^A$ represents the class $K$ up to a sign, the tautological form $ e_d^A$ represents the class $\kappa_d$ up to a sign, and the tautological form $h$ represents the class $\Delta$; 
\item the projections $p$, $p_i$ and $p_{ij}$ are tautological morphisms.
\end{itemize}
It follows from Looijenga's result in \cite{LooijengaTautological} that all tautological forms of degree larger than \( 2(g+r-2)\) are exact. We will show in this paper that certain $2$-forms obtained from natural normal functions on the moduli spaces $\cC_g^r$ are tautological (Theorem~\ref{thm:jacobian}), and in fact generate the tautological rings in a suitable sense (Theorem~\ref{thm:char_alt}). Also we show that  rings of  tautological forms are finite dimensional (Theorem~\ref{thm:finite_dim}). Finally,  we  describe a basis of the degree-two part $\tR^2(\cC_g^r)$ of the ring of tautological forms (Theorem~\ref{thm:degree_two}), as well as a basis of the space of exact tautological $2$-forms on $\cM_g$ (Theorem~\ref{thm:ZK_inv}).

 \subsection{Statement of the main results} \label{subsec:statement_main}
 
 \renewcommand*{\thethm}{\Alph{thm}}

We assume that $g \ge 2$.
Let $\cJ_g \to \cM_g$ denote the universal Jacobian. Let $\cP$ denote the Poincar\'e bundle on $\cJ_g \times_{\cM_g} \cJ^\lor_g$, equipped with the tautological rigidification along the projection onto the second coordinate. Let $\lambda \colon \cJ_g \imto \cJ_g^\lor$ denote the canonical principal polarization, and write $\cB$ for the rigidified line bundle $(\id,\lambda)^*\cP$ on $\cJ_g$. We remark that the restriction of $\cB$ to a fiber of $\cJ_g \to \cM_g$ represents \emph{twice} the principal polarization~$\lambda$.

By for instance \cite[Sections 6--7]{HainReed} the Chern class $c_1(\cB) \in H^2(\cJ_g,\R)$ contains a canonical $(1,1)$-form $2 \omega_0$, uniquely characterized by the following two properties:
\begin{itemize}
\item the form vanishes along the given rigidification;
\item the form is fiberwise translation invariant.
\end{itemize}
Following \cite[Section~2]{GeomBogom} we call $\omega_0$ the \emph{Betti form} on $\cJ_g$. 

Let \(r\geq 0\) be any integer, let $n \in \Z$, and let \(m = (m_1,\dots,m_r)\) be an \(r\)-tuple of integers whose sum equals~$(2g-2)n$. These data give rise to a natural map
\[ F_m \colon \cC_g^r \to \cJ_g  \]
given by sending an $r$-marked curve $(C,x_1,\ldots,x_r)$ to the class of the degree zero line bundle $\cO_C(m_1x_1 + \cdots + m_rx_r)\otimes \omega_C^{\otimes -n}$ in the Jacobian of~$C$. Here $\omega_C$ denotes the canonical line bundle on $C$.

Our first main result says that the canonical $2$-forms $F_m^*\omega_0$  are all tautological.
\begin{thm} \label{thm:jacobian} Each of the $2$-forms \(2F_m^*\omega_0\) is an integral linear combination of the tautological $2$-forms $p^*e_1^A$, $p_i^*e^A$ and $p_{ij}^*h$. 
\end{thm}
Let $2\phi_0$ denote the Chern class of $\cB$ in $H^2(\cJ_g,\Q)$. It follows from a result of Hain in \cite[Theorem~11.5]{HainNormal} that each of the classes $2F_m^* \phi_0$ is an integral linear combination of the tautological classes $p^*\kappa_1$, $p_i^*K$ and $p_{ij}^*\Delta$. Theorem~\ref{thm:jacobian} can be viewed as a refinement of Hain's result at the level of forms. Theorem~\ref{thm:jacobian} can in principle be proved by following the lines of \cite[Section~11]{HainNormal}, however we will follow here a slightly different route using Deligne pairings, a tool that we will need anyway. 

Let $\delta=F_{(1,-1)}$ denote the difference map $\cC_g^2 \to \cJ_g$. Our next result gives perhaps a more canonical description of the rings of tautological forms.
\begin{thm} \label{thm:char_alt} The rings of tautological forms $\tR^*(\cC_g^r)$ are the smallest $\R$-subalgebras of $A^*(\cC_g^r)$ stable under pullback along tautological morphisms and fiber integration along tautological submersions, and containing the  $2$-form   \(2\delta^*\omega_0\).
\end{thm}
For the proof, we need to show that we can obtain the form $h$ from the form $2\delta^*\omega_0$ using pullbacks and fiber integrations.

We next have the following finiteness result for our tautological rings.
 \begin{thm} \label{thm:finite_dim}
	For each \(r\geq 0\) and \(g\geq 2\), the ring of tautological forms \(\tR^*(\cC_g^r)\) is finite-dimensional as an $\R$-vector space.
\end{thm}
The proof proceeds in a few steps. First, we develop a graphical formalism that allows to attach tautological forms on $\cC_g^r$ to what we call \emph{$r$-marked graphs}. We will see that the forms associated to $r$-marked graphs span the tautological ring $\tR^*(\cC_g^r)$ as an $\R$-vector space. In fact, we shall see that it suffices to take only \emph{contracted} $r$-marked graphs; and these will be seen to be classified by a finite set.

Our next result gives an explicit description of the degree-two part $\tR^2(\cC_g^r)$ of the ring of tautological forms.
We set
\begin{equation} \label{eq:def_nu}
 \nu := \int_{\cC_g^2/\cM_g} h^3 \in \tR^2(\cM_g).
\end{equation}
\begin{thm} \label{thm:degree_two} Let $r \ge 0$. If \(g\geq 3\) a basis of $\tR^2(\cC_g^r)$ is given by the \(2\)-forms
	\[\{p_{ij}^* h: 1\leq i<j\leq r\} \cup \{p_i^*e^A: 1\leq i\leq r\} \cup \{e_1^A, \nu\}.\]
	If \(g=2\) a basis of $\tR^2(\cC_g^r)$ is given by the \(2\)-forms
	\[\{p_{ij}^* h: 1\leq i<j\leq r\} \cup \{p_i^*e^A: 1\leq i\leq r\} \cup \{e_1^A\}.\]
\end{thm}
We note that the forms $e_1^A$ and $\nu$ have the same class in cohomology. In particular, the difference $\nu - e_1^A$ is an example of an \emph{exact} tautological form. As it turns out, this exact form is intimately connected with the \emph{Kawazumi--Zhang invariant}  $\varphi \colon \cM_g \to \R$ introduced in \cites{kawazumi2008johnson, kawazumi2009canonical} and in \cite{zhang2010gross}, independently. Namely, the Kawazumi--Zhang invariant is given as the fiber integral
 \begin{equation} \label{eqn:def_ZK_inv}
  \varphi := \int_{\cC_g^2/\cM_g} \log G \, h^2  \in A^0(\cM_g) . 
  \end{equation}
A small calculation, using that $\partial \overline{\partial}$ commutes with fiber integration, shows that
\begin{equation} \label{eqn:lin_comb}
\nu - e_1^A = \frac{1}{\pi\i} \del\delbar\varphi \, . 
\end{equation}
We refer to \cite[Proposition~5.3]{dejongtorus} for details. Our final result says that the Kawazumi--Zhang invariant exactly explains all exact tautological $2$-forms over $\cM_g$.

\begin{thm} \label{thm:ZK_inv}	 The subspace of exact 2-forms in \( \tR^2(\cM_g)\) is one-dimensional,  spanned by the form $ \frac{1}{\pi \i} \partial \overline{\partial} \varphi = \nu -e_1^A$.
        When $g \ge 3$, the Kawazumi--Zhang invariant can be characterized as the only $C^\infty$-function on  \(\cM_g\), up to additive and multiplicative constants, whose Levi form is a tautological form.
\end{thm}

\subsection{Future directions}
Using a counting argument on contracted graphs, it is possible to show that for all $d \in \Z_{\ge 0}$ there exists a polynomial $f_{d}$ of degree $2d$ such that  for all $g \geq 2$, $r \geq 0$ the bound $\dim \tR^{2d}(\cC_g^r) \leq f_d(r)$ holds. The polynomials $f_d$ can in principle be computed. We refer to \cite{thesis} for details.

Theorem~\ref{thm:degree_two} shows that $\dim \tR^2(\cC_g^r)$ is essentially given by a quadratic polynomial in~$r$. Unfortunately, we have found it complicated to obtain precise information about (the growth behavior of the dimensions of) the $\tR^*(\cC_g^r)$ in higher cohomological degrees. 

A possible starting point might be to first obtain a better understanding of the homogeneous ideal $I^*(\cC_g^r) \subset \tR^*(\cC_g^r)$ of exact tautological forms, as our knowledge of the usual tautological rings gives us information about the quotient rings \(RH^*(\cC_g^r)\otimes_\Q \R\). For example, it would be interesting  to find generators for the ideals $I^*(\cC_g^r)$.

Other possible directions for future research could be to extend our constructions and results to the setting of \emph{cochains} for suitable mapping class groups, and to the setting of the moduli spaces of marked \emph{stable} curves.
 
 \renewcommand*{\thethm}{\arabic{section}.\arabic{thm}}

\subsection{Overview of the paper} 

Sections~\ref{sec:prelim}--\ref{sec:delignepairing} are used to set notations and to review known results. In Section~\ref{sec:canonical_h} we properly introduce the forms $h$ and $e^A$ and discuss some of their basic properties. In Section~\ref{sec:proof_AB} we prove Theorems~\ref{thm:jacobian} and~\ref{thm:char_alt}. In Sections~\ref{sec:marked_graphs}--\ref{sec:taut_forms_contr} we carry out foundational work for our proofs of Theorems~\ref{thm:finite_dim}, \ref{thm:degree_two} and~\ref{thm:ZK_inv}, which are then presented in Sections~\ref{sec:proof_fin_dim} and~\ref{sec:2forms}. \\

\noindent \textbf{Acknowledgments.} We thank Nariya Kawazumi for a question and discussion that initiated the research done for this paper. We thank Carel Faber, David Holmes and Dan Petersen for helpful remarks. The research for this project was financed by TOP grant 613.001.401 of the Dutch Research Council (NWO).

\section{Preliminaries} \label{sec:prelim}

In this paper we work with stacks over the category \textbf{CMan} of complex manifolds. For a thorough discussion of this notion, and of the results below, we refer to \cite[Chapter~2]{thesis} and the references therein.

\subsection{Stacks} Roughly speaking, a \emph{stack} (over \textbf{CMan}) is a category $\cX$ equipped with a functor $F \colon \cX \to \mathbf{CMan}$ that allows base changes, gluing of isomorphisms, and gluing of objects. A complex manifold $S$ becomes itself naturally a stack by considering the category of complex manifolds over $S$; the structure functor $F$ in this case is the functor that forgets the base manifold $S$. 

Stacks form a $2$-category. One has a natural notion of $2$-cartesian diagrams of stacks. When $\cX$, $\cS$ are stacks the category of morphisms from $\cX$ to $\cS$ is denoted $\cS(\cX)$. A morphism of stacks \(f:\cX\to\cS\) is called \emph{representable} if for each complex manifold $S$ and each morphism of stacks \(\Phi \colon S\to \cS\)  there exists a 2-cartesian diagram of the form
\[
	\begin{tikzcd}
		X \arrow[d] \arrow[r] \arrow[rd,phantom,"\square"] & \cX \arrow[d, "f"] \\
		S \arrow[r, "\Phi"']   & \cS               
	\end{tikzcd}
\]
with \(X\)  a complex manifold. 

Let P be a property of morphisms of complex manifolds that is compatible with base change. Then we say a morphism  \(f:\cX\to\cS\) of stacks has property P if it is representable and for each $2$-cartesian diagram as above the morphism of complex manifolds \(X\to S\) has property P. In particular, one can talk about   a morphism of stacks being a \emph{submersion}, being \emph{proper}, being \emph{surjective}, or being a \emph{family of Riemann surfaces}.

\subsection{Moduli stacks of curves} Stacks that are central in this paper are the stack $\cM_g$ of families $f \colon \cC \to S$ of compact connected Riemann surfaces of genus~$g$ and the stack $\cC_g$ of such families $f \colon \cC \to S$ together with a section $\sigma \colon S \to \cC$. The functor $p \colon \cC_g \to \cM_g$ that forgets the section is a representable morphism of stacks, called the \emph{universal family} of compact connected Riemann surfaces of genus~$g$. 

For each $r \in \Z_{>0}$ the stack $\cC_g^r$ is defined to be the $r$-fold fiber product of the morphism $p$ with itself. The tautological morphisms $f^\phi \colon \cC_g^r \to \cC_g^s$ considered in the introduction are proper morphisms of stacks. 

\subsection{Differential forms on stacks} We denote by $A^*$ the category whose objects are differentiable forms on a complex manifold. If \(\eta\) and \(\omega\) are differential forms on complex manifolds \(T\) and \(S\), respectively, then the morphisms \(\eta\to\omega\) in \(A^*\) are precisely those morphisms \(f:T\to S\) of the underlying manifolds for which \(f^*\omega = \eta\). 

We have a natural functor \(A^* \to \cat{CMan}\) given by sending a differential form to its underlying complex manifold. This functor turns \(A^*\) into a stack over \(\cat{CMan}\). 

When $\cX$ is a stack, a \emph{differential form} on $\cX$ is defined to be a morphism of stacks $\cX \to A^*$. For instance, differential forms on a complex manifold correspond bijectively to differential forms on the associated stack;
differential forms on the stack \(\cM_g\) are differential forms that occur universally on the bases of families of genus \(g\) compact Riemann surfaces. The category $A^*(\cX)$ is a \emph{discrete} category: there are no $2$-morphisms between two differential forms on a given stack, apart from identity morphisms. Hence the notion of equality of differential forms makes sense.

There are natural $d$-, $\partial$- and $\overline{\partial}$-operators on differential forms extending the usual ones, in particular we can talk about \emph{exact} and \emph{closed} differential forms. 

When  \(f:\cX\to\cS\) is a morphism of stacks, we immediately obtain a pullback functor $f^* \colon A^*(\cS) \to A^*(\cX)$. The next lemma says that the pullback functor is well-behaved with respect to submersions.
\begin{lem} \label{lem:submersion_inj} (See \cite[Lemma~2.5.4]{thesis})	Let \(f:\cX\to\cS\) be a surjective submersion of stacks. Then the functor \(f^*: A^*(\cS)\to A^*(\cX)\) is injective.
\end{lem}
When  \(f:\cX\to\cS\) is a proper submersion, we have a natural fiber integral functor $\int_f \colon A^*(\cX) \to A^*(\cS)$ generalizing the usual fiber integral operator \cite[Appendix~II]{Stoll}. In particular we have that the projection formula is satisfied: for all $\omega \in A^*(\cX)$ and $\eta \in A^*(\cS)$ we have the identity
\[ \int_f \left( \omega \wedge f^* \eta \right) = \left( \int_f \omega \right) \wedge \eta \]
in $A^*(\cS)$. Fiber integration satisfies the base change formula for $2$-cartesian diagrams.
\begin{lem} \label{lem:base_change_stacks} (See \cite[Proposition~2.5.9]{thesis})
	Consider a 2-cartesian diagram of stacks
	\[
		\begin{tikzcd}
			\cX' \arrow[d, "f'"'] \arrow[r, "h"] \arrow[rd, "\square", phantom] & \cX \arrow[d, "f"] \\
			\cS' \arrow[r, "g"']                                       & \cS               
		\end{tikzcd}
	\]
	where \(f\) and \(f'\) are proper submersions.
	Let \(\omega\) be a differential form on \(\cX\). Then the following identity holds in $A^*(\cS')$:
	\[g^*\left(\int_f \omega\right) = \int_{f'} h^*\omega.\]
\end{lem}

\subsection{Hermitian line bundles on stacks} Similarly, one has a stack $\cP ic$ of \emph{line bundles} on complex manifolds. Morphisms in $\cP ic$ are cartesian diagrams; when $\cX$ is a stack, a line bundle on $\cX$ is defined to be a morphism of stacks $\cX \to \cP ic$. For instance, a line bundle on the moduli stack \(\cM_g\) is a line bundle that occurs universally on the bases of families of genus \(g\) compact Riemann surfaces. When  \(f:\cX\to\cS\) is a morphism of stacks, we immediately obtain a pullback functor $f^* \colon \cP ic(\cS) \to \cP ic(\cX)$. 

In a very similar vein one has the stack $\overline{\cP ic}$ of \emph{Hermitian line bundles} on complex manifolds. The \emph{Chern form} is realized as a morphism of stacks $c_1 \colon  \overline{\cP ic} \to A^*$. 

\section{Arakelov-Green's function and Arakelov metric} \label{sec:arakmetric}

In this section we introduce the Arakelov-Green's function $G$ of a compact and connected Riemann surface. Also we introduce the Arakelov metric on the holomorphic cotangent line bundle. The main reference for this section is \cite{Arakelov1974}. 

Let $C$ be a compact and connected Riemann surface of genus $g$. We will assume in the sequel that $g \ge 1$. Denote by $\omega_C$ the holomorphic cotangent line bundle of $C$. Then on the space $\omega_C(C)$ of global sections we have a natural  Hermitian inner product, given by the prescription
\begin{equation} \label{defineinnerproduct} ( \eta,\eta' ) \mapsto \frac{\i}{2} \int_C \eta \wedge \bar{\eta}' \, . 
\end{equation}
Let $(\eta_1,\ldots,\eta_g)$ be an orthonormal basis of $\omega_C(C)$. The \emph{Arakelov $(1,1)$-form} of $C$ is defined to be the element
\begin{equation} \label{defArakvolume} \mu := \frac{\i}{2g} \sum_{j=1}^g \eta_j \wedge \overline{\eta}_j \in A^2(C) .
\end{equation}
It follows from the Riemann-Roch theorem that $\mu$ is a volume form on $C$; we clearly have $\int_C \mu =1$. 

When $P \in C$ is a point we denote by $\delta_P$ the Dirac delta current at $P$. The \emph{Arakelov-Green's function} of $C$ is the real-valued generalized function on $C \times C$ uniquely determined by the conditions
\begin{equation} \label{deldelbarg_Ar}
\partial \bar{\partial}_z \, \log G(P,z) =\pi \i \, (\mu(z) - \delta_P(z))
\end{equation}
and
\begin{equation} \label{normalization}
\int_C \log G(P,z) \,\mu(z) = 0
\end{equation}
for all $P  \in C$. An application of Stokes' theorem shows that one has a symmetry property
\begin{equation} \label{symmetry}
G(P,Q) = G(Q,P)
\end{equation}
for all $P$, $Q$ in $C$. Let $U \subset C$ be an open set and let $t \colon U \imto \mathbb{D}$ be a local coordinate where $\mathbb{D}$ denotes a small open disk in $\C$. Then we have a local expansion
\begin{equation}
\log G(P,Q) = \log |t(P) - t(Q)| + O(1)
\end{equation}
for all distinct $P$, $Q \in U$. In this expansion, the $O(1)$-term is a $C^\infty$ function depending on the choice of coordinate. 

Let $\Delta$ denote the diagonal on $C \times C$. The Arakelov-Green's function $G$ induces a natural Hermitian metric $\|\cdot \|$ on the holomorphic line bundle $\cO_{C \times C}(\Delta)$ on $C \times C$ by putting $  \|1\|(P,Q)=  G(P,Q)$ for $P, Q$ in $C$. Here $1$ denotes the canonical global section of $\cO_{C \times C}(\Delta)$. By restriction to vertical or horizontal slices of $C \times C$ we obtain natural Hermitian metrics on the line bundles $\cO_C(P)$ for each $P \in C$.

Let $\Delta$ also denote the diagonal embedding of $C$ into $C \times C$. Recall that we have a canonical  isomorphism
\begin{equation} \label{adjunction}\omega_C^{\otimes -1} \imto \Delta^* \cO_{C \times C}(\Delta)  
\end{equation}
of holomorphic line bundles on $C$ (the adjunction formula). 
By pullback along the isomorphism \eqref{adjunction} and taking the dual one obtains an induced $C^\infty$ metric on $\omega_C$, called the \emph{Arakelov metric}. 
\begin{definition}  \label{def:admissible} A Hermitian line bundle $(L,\|\cdot\|)$ on the Riemann surface $C$ is called \emph{admissible} if its Chern form $c_1(L,\|\cdot\|)$ is a multiple of the Arakelov $(1,1)$-form $\mu$.  
\end{definition}
Equation~\ref{deldelbarg_Ar} can be used to show that each $\cO_C(P)$ with its metric derived from $G$ is admissible. We also have that $\omega_C$ equipped with its Arakelov metric is admissible, as shown in \cite[Section~4]{Arakelov1974}. The dual of an admissible Hermitian line bundle is admissible, and the tensor product of two admissible Hermitian line bundles is admissible.

\section{Deligne pairing and its metric}   \label{sec:delignepairing} \label{sec:metrization}

Let $\cC$ and $S$ be complex manifolds. Let $p \colon \cC \to S$ be a family of compact Riemann surfaces of positive genus. Then following  \cite{Deligne} we have a canonical bi-multiplicative pairing $\pair{L,M}$ for line bundles $L, M$ on $\cC$, with values in line bundles  on $S$. Here and in the following, line bundles are always taken in the holomorphic category. As our construction will show, the formation of the Deligne pairing is compatible with base change. It follows that the notion of Deligne pairing generalizes to the context of families of compact Riemann surfaces over stacks. 

Locally on an open set $U \subset S$ the line bundle $\pair{L,M}$ is generated by symbols $\pair{\ell,m}$ with $\ell$ a nonzero rational section of $L$ on $p^{-1}U$ and $m$ a nonzero rational section of $M$ on $p^{-1}U$, such that the divisors of $\ell, m$ on $p^{-1}U$ have disjoint support. These symbols obey the relations
\begin{equation} \label{relations} \pair{\ell, fm } = f(\divisor \ell)\pair{\ell,m} \, , \quad
\pair{f\ell,m} = f(\divisor m) \pair{\ell, m}  
\end{equation}
for rational functions $f$ on $\cC$. The function $f(\divisor \ell)$ should be interpreted as coming from a norm: when $D$ is an effective relative Cartier divisor on $\cC$, then we put $f(D)=\mathrm{Nm}_{D/S}(f)$. The Weil reciprocity law $f(\divisor g)=g(\divisor f)$ on compact Riemann surfaces can be used to show that this construction by generators and relations indeed gives a line bundle on $S$.  Let $a, b \in H^2(\cC,\Q)$ be the Chern classes of the line bundles $L, M$ on $\cC$. Then the Chern class of $\pair{L,M}$ is equal to $p_*(a \cup b) \in H^2(S,\Q)$.

Let $P \colon S \to \cC$ be a section of $p$. Let $L, M$ be line bundles on $\cC$, and let $N$ be a line bundle on $S$. Then we have canonical isomorphisms
\begin{equation} \label{canisoms} \pair{M,p^*N} \imto N^{\otimes \deg M} \, , \quad \pair{L,M} \imto \pair{M,L} \, , \quad \pair{\cO_\cC(P),L} \imto P^* L 
\end{equation}
of line bundles on $S$. 

Assume now that each of $L, M, N$ are equipped with Hermitian metrics. Then by \cite[Section 6]{Deligne} the Deligne pairing $\pair{L,M}$ has a canonical structure of Hermitian line bundle, which can be given explicitly as follows. Let $\ell$, $m$ be non-zero rational sections of $L$ resp.\ $M$ with disjoint support. Then we set
\begin{equation} \label{defmetricpairing}
\log \| \pair{\ell,m} \|  := (\log\|m\|) [\divisor \ell] + \int_p \log \|\ell\| \, c_1(M)
\end{equation}
as functions on $S$.
We have a symmetry relation $\|\pair{\ell,m}\|=\|\pair{m,\ell}\|$, cf. \cite[Section~6.3]{Deligne}. It follows from (\ref{defmetricpairing}) that for rational functions $f$ on $\cC$ we have
\[ \log\|\pair{\ell,fm}\| = \log \|\pair{\ell,m}\| + (\log|f|)[\divisor \ell]  ,  \]
showing that the norm $\|\cdot \|$ is compatible with the relations (\ref{relations}). 

The canonical isomorphisms
\begin{equation} \label{delignenorm} \pair{M,p^*N} \imto N^{\otimes \deg M} \, , \quad \pair{L,M} \imto \pair{M,L}  
\end{equation}
from (\ref{canisoms}) are easily seen to be isometries. For the third isomorphism from (\ref{canisoms}) we have to be a little careful. First of all, we equip $\cO_\cC(P)$ with the Hermitian metric derived from the Arakelov Green's function in the fibers. 
\begin{prop} \label{prop:fiberwise} Assume that the Hermitian line bundle $L$ is fiberwise admissible (in the sense of Definition~\ref{def:admissible}) with respect to~$p$. Then the canonical isomorphism $\pair{\cO_\cC(P),L} \imto P^* L $ from (\ref{canisoms}) is an isometry.
\end{prop}
\begin{proof} Denote by $1_P$ the canonical global section of $\cO_\cC(P)$. Let $\ell$ be any nonzero rational section of $L$ with support away from $P$. Then by the definition of the metric on the Deligne pairing in (\ref{defmetricpairing}) and by the normalization condition (\ref{normalization}) we find
\begin{equation} \begin{split}
\log \| \pair{1_P,\ell} \| & = (\log \|\ell\|)[\divisor 1_P] + \int_p \log \|1_P\| \, c_1(L) \\
 & = \log \|P^*\ell\| +  \int_p \log G(P,-) \, c_1(L) \\
 & =  \log \|P^*\ell\| . \end{split} \end{equation}
The proposition follows.
\end{proof}
We note that $\cO_\cC(P)$ is itself fiberwise admissible with respect to~$p$. Also the relative cotangent bundle $\omega$ of $\cC$ over $S$ is fiberwise admissible with respect to~$p$, if one equips $\omega$ with the fiberwise Arakelov metric. 

We have the following useful expression for the Chern form of the Deligne pairing. 
\begin{prop} \label{c1deligne}
Let $p \colon \cC \to S$ be a family of compact Riemann surfaces, and $L$ and $M$ two Hermitian line bundles on $\cC$. Let $\langle L, M \rangle$ be the Deligne pairing of $L, M$ along $p$, equipped with its Hermitian metric determined by (\ref{defmetricpairing}). Then the equality of differential forms
\[ c_1( \langle L, M \rangle) = \int_p c_1(L) \wedge c_1(M)  \]
holds in $A^2(S)$.
\end{prop}
\begin{proof} This is \cite[Proposition 6.6]{Deligne}.
\end{proof}
Finally, we briefly discuss the connection with the Poincar\'e bundle on the Jacobian. We refer to \cite{MB} for an extensive discussion of this connection. 

Let $\cJ \to S$ denote the family of Jacobians associated to the family of compact Riemann surfaces $\cC \to S$. Let $\cP$ denote the Poincar\'e bundle on $\cJ \times_S \cJ^\lor$, equipped with its tautological rigidification along the zero section of the projection on the second coordinate. Let $\lambda \colon \cJ \imto \cJ^\lor$ denote the canonical principal polarization, and write $\cP_0$ for the rigidified line bundle $(\id \times \lambda)^*\cP$ on $\cJ \times_S \cJ$. 

The line bundle $\cP$ carries a canonical Hermitian metric, uniquely characterized by the following properties:
\begin{itemize}
\item the metric is compatible with the given rigidification;
\item the Chern form of the metric is translation invariant in all fibers over $S$.
\end{itemize}
Here, the norm on the trivial line bundle is taken to be the canonical one with $\|1\|=1$.
The canonical Hermitian metric on $\cP$ induces by pullback along $(\id \times \lambda)$ a  Hermitian metric on the line bundle $\cP_0$. 

We have the following fundamental result that we shall use in our proof of Theorem~\ref{thm:jacobian}.
\begin{thm} \label{thm:deligne_poinc} Let $L, M$ be two fiberwise admissible line bundles on $\cC$. Assume that $L, M$ have relative degree zero. Let $[L], [M]$ denote the resulting sections of the Jacobian fibration $\cJ \to S$. We have a canonical isometry of Hermitian line bundles
\[ \left( [L], [M] \right)^* \cP_0^{\otimes -1} \imto \pair{L,M} \]
on $S$. Here, the left hand side is equipped with the metric induced by pullback from the canonical metric on $\cP_0$.
\end{thm}
\begin{proof} This follows from \cite[Corollaire 4.14.1]{MB}.
\end{proof}

\section{The universal case} \label{sec:canonical_h}

Let $g \ge 2$ be an integer. The constructions and results from Sections~\ref{sec:arakmetric} and~\ref{sec:delignepairing} generalize to the setting of the universal compact Riemann surface $p \colon \cC_g \to \cM_g$, following the general remarks in Section~\ref{sec:prelim}. To start with, we have a natural Hermitian metric $\|\cdot\|$ on the  line bundle \(\cO(\Delta)\) on \(\cC_g^2\), by setting $\|1\|(P,Q)=G(P,Q)$ for a pair of points $P, Q$ on a compact Riemann surface $C$, with $G$ the Arakelov-Green's function of $C$.
\begin{definition}
We set
 \begin{equation}
 h:=c_1(\cO(\Delta), \|\cdot\|) \in A^2(\cC_g^2),
 \end{equation}
 the Chern form of the Hermitian line bundle \(\cO(\Delta)\) on \(\cC_g^2\). 
 The $2$-form $h$ represents the class $\Delta$ of the diagonal in $H^2(\cC_g^2,\R)$. 
  \end{definition}
 Let \(\omega = \omega_{\cC_g/\cM_g}\) denote the relative holomorphic cotangent bundle of the family of compact Riemann surfaces \(p \colon \cC_g \to \cM_g\). We endow $\omega$ with the fiberwise Arakelov metric; this turns $\omega$ into a Hermitian line bundle on the stack $\cC_g$. Let $\Delta \colon \cC_g \to \cC_g^2$ also denote the diagonal embedding. Then by construction of the Arakelov metric via the adjunction formula \eqref{adjunction} we arrive at a canonical isometry
\begin{equation} \label{adj_isometry}
\omega^{\otimes-1} \simeq \Delta^*\cO(\Delta)
\end{equation}
of Hermitian line bundles on \(\cC_g\). 
\begin{definition}
We set
\begin{equation}
e^A := c_1(\omega^{\otimes -1}, \|\cdot\|) \in A^2(\cC_g),
\end{equation}
the Chern form of $\omega^{-1}$ equipped with the dual of the Arakelov metric. The $2$-form $-e^A$ represents the  class \(K\) of $\omega$ in \(H^2(\cC_g,\R)\), and we have $e^A = \Delta^*h$. 
\end{definition}
Let \(p_1:\cC_g^2\to\cC_g\) be the projection on the first coordinate. One readily finds the identities
\begin{equation} \label{lem:inteA} 
\int_p e^A = 2-2g \in A^0(\cM_g), \qquad \int_{p_1} h = 1 \in A^0(\cC_g).
\end{equation}
We will also need the following results. 
\begin{lemma} \label{lem:inteAc1L}
	Consider the family of compact Riemann surfaces \(p_1:\cC_g^2\to\cC_g\).
	If \(L\) is a Hermitian line bundle on \(\cC_g^2\) which is fiberwise admissible with respect to \(p_1\), then 
	\[\int_{p_1} h \wedge c_1(L) = \Delta^* c_1(L) \in A^2(\cC_g).\]
	In particular, we have:
	\[\int_{p_1} h^2 = e^A\]
	and for \(i=1,2\) we have
	\[\int_{p_1} h \wedge p_i^* e^A = e^A.\]
\end{lemma}
\begin{proof}
	From Proposition~\ref{c1deligne} we obtain
	\[
		\int_{p_1} h \wedge c_1(L) = \int_{p_1} c_1(\cO(\Delta))\wedge c_1(L) = c_1(\inp{\cO(\Delta),L}) = c_1(\Delta^* L) = \Delta^* c_1(L),
	\]
	where the third equality follows Proposition~\ref{prop:fiberwise}.
	The other identities now follow from:
	\[h = c_1(\cO(\Delta)),\quad \text{and}\quad p_i^*e^A = p_i^*c_1(\omega^{\otimes-1}) = c_1(p_i^*\omega^{\otimes-1}).\qedhere\]
\end{proof}

\begin{lemma} \label{lem:p13hp23h}
	Let \(p_{12},p_{13},p_{23}:\cC_g^3\to\cC_g^2\) be the three projections.
	Then 
	\[\int_{p_{12}} p_{13}^*h \wedge p_{23}^* h = h \in A^2(\cC_g^2).\]
\end{lemma}
\begin{proof}
	Let \(\sigma_1,\sigma_2:\cC_g^2\to \cC_g^3\) be the two canonical sections of \(p_{12}\), such that \(p_3\circ \sigma_i = p_i: \cC_g^2\to\cC_g\) for \(i=1,2\).
	Notice that \(p_{13}\circ \sigma_2:\cC_g^2\to\cC_g^2\) is the identity.
	Endow the induced line bundles \(\cO(\sigma_1), \cO(\sigma_2)\) on \(\cC_g^3\) with their canonical metrics.
	We use Proposition~\ref{c1deligne}  to obtain
	\begin{align*}
		\int_{p_{12}} p_{13}^* h \wedge p_{23}^* h &= \int_{p_{12}} p_{13}^* c_1(\cO(\Delta)) \wedge p_{23}^* c_1(\cO(\Delta))\\
							   &= \int_{p_{12}} c_1(\cO(\sigma_1)) \wedge c_1(\cO(\sigma_2)) \\
							   &= c_1(\inp{\cO(\sigma_1),\cO(\sigma_2)}) \\
							   &= \sigma_2^* c_1(\cO(\sigma_1))\\
							   &= \sigma_2^* p_{13}^* c_1(\cO(\Delta))\\
							   &= c_1(\cO(\Delta))\\
							   &= h.\qedhere
	\end{align*}
\end{proof}

\begin{definition}
We set
\begin{equation}
e_1^A := \int_p (e^A)^2  \in A^2(\cM_g).
\end{equation}
Note that $e_1^A$ equals the Chern form of $\pair{\omega,\omega}$, when the latter is equipped with the Deligne pairing metric, by Proposition~\ref{c1deligne}.  The $2$-form $e_1^A$ represents the kappa class \(\kappa_1\) in \(H^2(\cM_g,\R)\). 
\end{definition}

\section{Proof of Theorems~\ref{thm:jacobian} and~\ref{thm:char_alt}} \label{sec:proof_AB}

We continue to assume that $g \ge 2$. 
Let $\cJ_g \to \cM_g$ denote the universal Jacobian in genus~$g$, i.e., the Jacobian fibration associated to the universal family  $p \colon \cC_g \to \cM_g$. Let $\cP$ denote the Poincar\'e bundle on $\cJ_g \times_{\cM_g} \cJ^\lor_g$, equipped with its tautological rigidification along the zero section of the projection on the second coordinate. Let $\lambda \colon \cJ_g \imto \cJ^\lor_g$ denote the canonical principal polarization, write $\cP_0$ for the rigidified line bundle $(\id,\lambda)^*\cP$ on $\cJ_g \times_{\cM_g} \cJ_g$ and write $\cB$ for the pullback of the line bundle $\cP_0$ along the diagonal. This is a rigidified line bundle on $\cJ_g$.

Following the discussion at the end of Section~\ref{sec:delignepairing}, we see that the line bundle $\cP$ carries a canonical  Hermitian metric, uniquely characterized by the following properties:
\begin{itemize}
\item the metric is compatible with the given rigidification;
\item the Chern form of the metric is translation invariant in all fibers over $\cM_g$.
\end{itemize}
The canonical Hermitian metric on $\cP$ induces by pullback a Hermitian metric on the line bundles $\cP_0$ and $\cB$. We observe that the Chern form of $\cB$ is equal to the form $2 \omega_0$ on $\cJ_g$ that we introduced in \ref{subsec:statement_main}. 

Let \(r\geq 1\) be any non-negative integer, let $n$ be any integer, and let \(m = (m_1,\dots,m_r)\) be an \(r\)-tuple of integers whose sum equals $(2g-2)n$.
These data give rise to a morphism of stacks
\begin{equation} \label{eqn:F_m}
 F_m \colon \cC_g^r \to \cJ_g 
\end{equation}
given by sending a family $p \colon \cC \to S$ of compact Riemann surfaces of genus $g$ together with an $r$-tuple of sections $(\sigma_1,\ldots,\sigma_r)$  to the section of the associated Jacobian fibration over $S$ given by the relative degree zero line bundle $\cO_\cC(m_1\sigma_1+\cdots+m_r\sigma_r) \otimes \omega_{\cC/S}^{\otimes -n}$ on~$\cC$. Here $\omega_{\cC/S}$ denotes the relative holomorphic cotangent bundle of $\cC$ over $S$.

\begin{thm} \label{prop:fmisometry} Let $\omega=\omega_{\cC_g/\cM_g}$ denote the relative holomorphic cotangent bundle of the universal family of compact Riemann surfaces of genus~$g$, equipped with the fiberwise Arakelov metric.
We have a canonical isometry of Hermitian line bundles on \(\cC_g^r\):
	\[ F_m^*\cB^{\otimes-1} \simeq \bigotimes_{1\leq i<j\leq r} p_{ij}^* \cO(\Delta)^{\otimes 2 m_i m_j} \otimes \bigotimes_{i=1}^r p_i^*\omega^{\otimes-m_i^2 - 2m_i n} \otimes \pair{ \omega, \omega}^{\otimes n^2}. \]
\end{thm}
\begin{proof} We temporarily write $p \colon \cC_g^{r+1} \to \cC_g^r$ for the projection forgetting the last coordinate. We have canonical sections $\sigma_i \colon \cC_g^r \to \cC_g^{r+1}$ of $p$ for $i=1,\ldots, r$ obtained by repeating the $i$-th coordinate. By a slight abuse of notation we write $\omega$ also for the relative cotangent bundle of $p$. We then set
\[ L_m = \cO(m_1\sigma_1 + \cdots + m_r\sigma_r) \otimes \omega^{\otimes -n} , \]
a line bundle on $\cC_g^{r+1}$ of relative degree zero over $\cC_g^r$. Following the constructions and results from 
Sections~\ref{sec:prelim}--\ref{sec:delignepairing} we have a natural Hermitian metric on $L_m$, obtained from the metric determined by the Arakelov-Green's function on the line bundles $\cO(\sigma_i)$ and from the fiberwise Arakelov metric on~$\omega$. The Hermitian line bundle $L_m$ is fiberwise admissible. By Theorem~\ref{thm:deligne_poinc} we have  a canonical isometry
\[F_m^*\cB^{\otimes-1} \isoto \inp{L_m, L_m},\]
where the Deligne pairing is taken along~$p$.
As the Deligne pairing is bimultiplicative, we find another canonical isometry
\[\inp{L_m,L_m} \isoto \bigotimes_{i=1}^r\bigotimes_{j=1}^r \inp{\cO(\sigma_i),\cO(\sigma_j)}^{\otimes m_i m_j} \otimes \bigotimes_{i=1}^r p_i^* \omega^{\otimes -2nm_i} \otimes \pair{\omega,\omega}^{\otimes n^2} .\]
Now for all \(1\leq j\leq r\) we have a canonical isometry
\[\cO(\sigma_j) \isoto p_{j,r+1}^* \cO(\Delta)\]
so taking the pullback along \(\sigma_i\) yields a string of canonical isometries
\[ \inp{\cO(\sigma_i),\cO(\sigma_j)} \simeq \sigma_i^* \cO(\sigma_j) \simeq \sigma_i^* p_{j,r+1}^* \cO(\Delta) \simeq \begin{cases} p_{ji}^*\cO(\Delta) = p_{ij}^*\cO(\Delta) &\text{if }i\neq j\\p_j^*\Delta^*\cO(\Delta) = p_j^* \omega^{\otimes-1} &\text{if }i=j.\end{cases}\]
Here the first isometry follows from Proposition~\ref{prop:fiberwise}. In the last equality we have used the isometry~\eqref{adj_isometry}.
By combining the above canonical isometries we obtain the result.
\end{proof}
The following corollary has Theorem~\ref{thm:jacobian} as an immediate consequence. It is a refinement of \cite[Theorem~11.5]{HainNormal} at the level of forms.
\begin{cor} \label{cor:pullbackomega}
We have the following equality of 2-forms on \(\cC_g^r\):
	\[-2F_m^*\omega_0 = \sum_{1\leq i<j\leq r} 2m_im_j p_{ij}^* h + \sum_{i=1}^r (m_i^2 + 2m_i n) p_i^* e^A + n^2  e^A_1 \in A^2(\cC_g^r).\]
\end{cor}
\begin{proof} Simply take Chern forms on left and right hand side in Theorem~\ref{prop:fmisometry}.
\end{proof}
A key example is obtained by setting $m=(1,-1)$. Writing $\delta=F_{(1,-1)}$ we obtain from Corollary~\ref{cor:pullbackomega}  the  identity
\begin{equation} \label{eqn:pullback_diff}
-2\delta^*\omega_0 = -2h + p_1^*e^A + p_2^*e^A 
\end{equation}
in $A^2(\cC_g^2)$. See also \cite[Theorem~1.4]{dejongtorus}.
\begin{proof}[Proof of Theorem~\ref{thm:char_alt}]
We need to show that we can obtain the $2$-form \(h\) from the $2$-form \(2\delta^*\omega_0\) by using pullbacks and fiber integrals. This is not difficult using the results from Section~\ref{sec:canonical_h}; we refer to \cite[Section~4.3]{thesis} for details of the following computation. 
Squaring left and right hand side of the identity in (\ref{eqn:pullback_diff}) and integrating the result along the fibers of \(p_1:\cC_g^2\to\cC_g\)  first of all yields
\[ \int_{p_1}(-2\delta^*\omega_0)^2 = -4ge^A + p^*e_1^A. \]
Next, squaring the latter form and integrating it along the fibers of \(p:\cC_g\to\cM_g\) gives
\[ \int_p (-4ge^A + p^*e_1^A)^2 =16g(2g-1) e_1^A.\]
We conclude that we can obtain \(e_1^A\), then \(e^A\), and finally \(h\) from \(2\delta^*\omega_0\) by taking fiber integrals and pullbacks.
\end{proof}
\begin{remark}
It is not hard to see using the explicit description of the Betti form $\omega_0$ in e.g.\  \cite[Section~2]{GeomBogom}  that the $(2g+2)$-form $ \omega_0^{g+1}  $ \emph{vanishes identically} on $\cJ_g$. We see that raising the right hand side in Corollary~\ref{cor:pullbackomega} to the $(g+1)$-st power gives rise to a polynomial relation among the tautological forms $p_{ij}^*h$, $p_i^*e^A$ and $e_1^A$; one gets further relations by fiber integrating these down to $\cC_g^r$'s with smaller $r$, and/or by using the fact that we get a polynomial in the \emph{variables} $m_i$ that vanishes identically. This method to generate relations among tautological forms refines a powerful idea due to Randal-Williams \cite{RandalWilliams} to obtain relations among tautological classes in cohomology.  

We present here only one example: setting $g=2$, and working with $2\delta^*\omega_0$ we get from Equation \ref{eqn:pullback_diff} the identity
\[ \left( -2h + p_1^*e^A + p_2^*e^A \right)^3 = 0 \in \tR^6(\cC_2^2) \, . \]
Expanding the parentheses, and fiber integrating the resulting ten tautological $6$-forms down to $\cM_2$ turns out to yield the interesting relation
\begin{equation} \label{eq:special_rel_g=2}
 8\nu + 12 e_1^A  = 0 \in \tR^2(\cM_2) ,
\end{equation}
where $\nu$ is the form defined in Equation~\ref{eq:def_nu}.
We refer to \cite[Section~4.10]{thesis} for details of this computation, and for a further discussion of Randal-Williams' method for forms.
\end{remark}

\begin{remark} \label{rem:normal_functions} Each of the morphisms $F_m$ from Equation \ref{eqn:F_m} is an example of a ``normal function'' on $\cC_g^r$  in the sense of Hodge theory. In \cite{HainNormal} Hain considers, apart from the normal functions $F_m$, also a certain normal function on $\cM_g$ related to the \emph{Ceresa cycle}. The analogue of the line bundle $F_m^*\cB$ above is the so-called \emph{Hain-Reed line bundle}, studied extensively in \cite{HainReed}. By \cite[Proposition~10.2]{dejongtorus} the Chern form of the Hain-Reed line bundle is proportional to the $2$-form called $e_1^J$ in \cites{kawazumi2008johnson,  kawazumi2009canonical}. As follows from \cite[Theorem~1.4]{dejongtorus}, the $2$-form $e_1^J$ on $\cM_g$ is a linear combination of the $2$-forms $e_1^A$ and  \(\frac{1}{\pi\i} \del\delbar\varphi\), and hence is a tautological $2$-form. 

Now \cite[Theorem~A.1]{HainNormal} states, in a very rough form, that the normal functions $F_m$ and $\nu$ are essentially the \emph{only} normal functions on the moduli spaces $\cC_g^r$ that satisfy the property that their monodromy representation factors through a rational representation of $\Sp_{2g}$. We find that Theorem~\ref{thm:char_alt} can be rephrased as saying that the system of rings of tautological forms can be characterized  as the smallest system of $\R$-algebras of differential forms that is closed under tautological pullbacks and submersions, and contains all natural $2$-forms obtained from the normal functions whose monodromy representations factor through a rational representation of $\Sp_{2g}$. We mention that the corresponding result in cohomology is implicit in the work of Kawazumi and Morita \cite{KawazumiMorita} and Petersen, Tavakol and Yin \cite{PTY}.
\end{remark}

\section{Marked graphs} \label{sec:marked_graphs}

The purpose of Sections~\ref{sec:marked_graphs}--\ref{sec:taut_forms_contr} is to prepare for the proofs of Theorems~\ref{thm:finite_dim}--\ref{thm:ZK_inv}.

\subsection{The category of \texorpdfstring{\(r\)-marked}{r-marked} graphs}
In this paper, a \emph{graph} is a pair \((V,E)\), consisting of a finite set \(V\) of \emph{vertices}, and a finite multiset \(E\) of \emph{edges} consisting of unordered pairs (multisets of cardinality 2) of elements of \(V\). If \(e\in E\) is an edge, its two elements are called the \emph{endpoints} of \(e\). If these endpoints are the same, we call \(e\) a \emph{loop}. The \emph{degree} of a vertex \(v\in V\), denoted \(\deg v\), is the number of times \(v\) occurs as an endpoint of an edge of \(E\); that is: the multiplicity of \(v\) in the multiset sum of all edges \(e\in E\). In particular, we see that each loop contributes 2 to the degree of the vertex it is based on.

If \(\Gamma = (V,E)\) is a graph, then the \emph{(Euler) characteristic} of \(\Gamma\) is defined as
\[\chi(\Gamma) = \abs{V} - \abs{E}.\]
The Euler characteristic is additive on disjoint unions of graphs. 

Let \(r\geq 0\) be an integer. An \emph{\(r\)-marked graph} \((V,E,m)\) is a graph \(\Gamma = (V,E)\) equipped with a \emph{marking} \(m\); that is: an injective map \(m: \{1,\dots,r\} \to V\). So a marked graph can be seen as a graph of which \(r\) vertices are labeled \(1,\dots,r\). An \emph{unmarked graph} is a \(0\)-marked graph, which is the same as an `ordinary' graph. 

Let \(\Gamma = (V,E,m)\) be an \(r\)-marked graph.
A vertex \(v\in V\) is \emph{marked} if it is in the image of \(m\), and \emph{unmarked} otherwise. We have a partition of \(V\) in a subset \(V_+\) of marked vertices and a subset \(V_-\) of unmarked vertices. 

Let \(\Gamma=(V,E,m)\) and \(\Gamma'=(V',E',m')\) be two \(r\)-marked graphs. A \emph{morphism of \(r\)-marked graphs} \(f:\Gamma \to \Gamma'\) is a pair of maps \((f_\vv: V\to V', f_\ee: E\to E')\), such that \(f_\vv\) respects the \(r\)-marking (that is: \(f_\vv\circ m = m'\)), and such that for each edge \(e\in E\) with endpoints \(v,w\), the edge \(f_\ee(e)\in E'\) has endpoints \(f_\vv(v)\) and \(f_\vv(w)\). 

We obtain a category \(\cG_r\) of \(r\)-marked graphs. Two \(r\)-marked graphs \(\Gamma\) and \(\Gamma'\) are isomorphic if and only if there exists a bijection on vertices that respects the markings of \(\Gamma\) and \(\Gamma'\), such that for each pair of vertices \(v,w\) of \(\Gamma\) the number of edges between \(v\) and \(w\) equals the number of edges between the corresponding vertices of \(\Gamma'\).

Assume that \(\Gamma=(V,E)\) is a graph, and let \(f: V \to V'\) be a map of finite sets. The \emph{graph induced from \(\Gamma\) by \(f\)}, notation \(\Gamma_f\), is the graph \((V',E')\) with set of vertices equal to \(V'\), and with edges
\[E' = \{\{f(v_1),f(v_2)\}: \{v_1,v_2\}\in E\}.\]
Notice that in particular we have \(\abs{E} = \abs{E'}\). 

The characteristic of \(\Gamma_f\) equals
\[\chi(\Gamma_f) = \chi(\Gamma) + \abs{V'} - \abs{V}.\]
If \(v'\in V'\) is a vertex in \(V'\), its degree is given by:
\[\deg(v') = \sum_{v\in f^{-1}(v')} \deg(v).\]

\subsection{Gluing marked graphs}
In this section we define a binary operation \(\sqcup_{r}\) on the category of \(r\)-marked graphs \(\cG_r\). It turns out that \(\sqcup_{r}\) is the coproduct in the category \(\cG_r\). We define the binary operation \(\sqcup_{r}\) on two \(r\)-marked graphs \(\Gamma,\Gamma'\) by gluing their marked vertices pairwise. More precisely, we proceed as follows.

Let \(\Gamma = (V,E,m)\) and \(\Gamma' = (V',E',m')\) be two \(r\)-marked graphs, and let \(\Gamma \sqcup \Gamma' = (V\sqcup V', E+E')\) denote the disjoint union of the underlying (unmarked) graphs. Consider the set \(V''\) defined by the pushout diagram
\begin{equation}\label{eqn:pushoutrglue}
	\begin{tikzcd}
		\{1,\dots,r\} \rar{m} \dar{m'} \arrow[dr,phantom,"\lrcorner", near end] & V \dar \\
		V' \rar & V''.
	\end{tikzcd}
\end{equation}
In other words, \(V''\) is the set \((V\sqcup V')/\sim\), where \(\sim\) is the smallest equivalence relation on \(V\sqcup V'\) such that \(m(i) \sim m'(i)\) for all \(i\in\{1,\dots,r\}\). Note, moreover, that the map \(m'':\{1,\dots,r\} \to V''\) induced by the above diagram is injective, since \(m\) and \(m'\) are injective.
\begin{definition}
	The \(r\)-marked graph \(\Gamma\sqcup_{r}\Gamma'\) is the graph induced from the disjoint union \(\Gamma\sqcup \Gamma'\) by the natural map \(V\sqcup V'\to V''\), endowed with the \(r\)-marking \(m'':\{1,\dots,r\}\to V''\) obtained from pushout diagram \ref{eqn:pushoutrglue}. 
\end{definition}

Suppose that \(\Gamma\) has \(u\) unmarked vertices and \(e\) edges, and that \(\Gamma'\) has \(u'\) unmarked vertices and \(e'\) edges. It follows that \(\Gamma\sqcup_{r}\Gamma'\) has \(u+u'\) unmarked vertices and \(e+e'\) edges. Therefore, the characteristic of \(\Gamma \sqcup_{r}\Gamma'\) is given by
\begin{equation}\label{eqn:charvsglue}
	\chi(\Gamma\sqcup_{r}\Gamma') = \chi(\Gamma) + \chi(\Gamma') - r.
\end{equation}

The set of vertices of \(\Gamma\sqcup_{r} \Gamma'\) is the pushout of the maps \(m\) and \(m'\). 
The operator \(\sqcup_{0}\) on \(\cG_0\) is simply the disjoint union. On \(\cG_1\) the operator \(\sqcup_{1}\) is the wedge sum. 

\subsection{Pushforward maps on marked graphs}\label{sec:pushforward}
Let \(\phi: \{1,\dots,s\} \to \{1,\dots,r\}\) be a map of sets. We will define a pushforward functor \(\phi_*: \cG_s \to \cG_r\). Given a graph \(\Gamma \in \cG_s\), the pushforward \(\phi_*\Gamma\) is obtained from \(\Gamma\) by replacing the \(s\) marked vertices by \(r\) marked vertices, as follows. 

Let \(\Gamma = (V,E,m)\) be an \(s\)-marked graph. Consider the pushout diagram (of sets)

\begin{equation}\label{diag:pushout}
\begin{tikzcd}
	\{1,\dots,s\} \rar{m}\ar[phantom,dr,"\lrcorner", near end] \dar{\phi} & V \dar{\phi_V} \\
	\{1,\dots,r\} \rar["m'"']{} & V'.
\end{tikzcd}
\end{equation}

As \(m\) is injective, it follows that \(m'\) must be injective. 

We define \(\phi_*\Gamma\) to be the graph \((V',E',m')\), where \((V',E')\) is the graph induced from \((V,E)\) by \(\phi_V\), and \(m'\) is the map defined in diagram \ref{diag:pushout}. 
Notice that \(\phi_V\) then induces a bijection between the unmarked vertices of \(\Gamma\) and \(\phi_*\Gamma\).

The characteristic of \(\phi_*\Gamma\) is given by
\[\chi(\phi_*\Gamma) = \chi(\Gamma) - s + r.\]

Moreover, if \(f:\Gamma_1 \to \Gamma_2\) is a morphism of \(s\)-marked graphs, we obtain an induced morphism of \(r\)-marked graphs \(\phi_*f: \phi_*\Gamma_1 \to \phi_*\Gamma_2\), via the universal property of the pushout diagram \ref{diag:pushout}. We obtain a covariant functor
\[\phi_*: \cG_s \to \cG_r.\]

The following properties of  the pushforward functor are not hard to show.
\begin{prop} \label{prop:pushforwardcomposition}
	Let \(\phi: \{1,\dots,s\} \to \{1,\dots,r\}\) and \(\psi: \{1,\dots, t\} \to \{1,\dots,s\}\) be maps. Then the functors \(\phi_*\psi_*\) and \((\phi\psi)_*\) from \(\cG_t\) to \(\cG_r\) are naturally isomorphic. \qed
\end{prop}
\begin{prop} \label{prop:pushfwdglue}
	Let \(\phi: \{1,\dots,s\}\to \{1,\dots,r\}\) be a map.
	Let \(\Gamma\) and \(\Gamma'\) be two \(s\)-marked graphs.
	Then there is a canonical isomorphism of graphs
	\[\phi_*(\Gamma\sqcup_s\Gamma') \simeq \phi_*(\Gamma) \sqcup_r \phi_*(\Gamma').\]
\end{prop}

\subsection{Pullback maps on marked graphs}
The next operation we will consider is a pullback operation. Let \[\phi:\{1,\dots,s\} \to \{1,\dots,r\}\] be an \emph{injective} map. Then we define a \emph{pullback} functor
\[\phi^*: \cG_r \to \cG_s\]
as follows. For any \(r\)-marked graph \(\Gamma = (V,E,m)\) we simply define \(\phi^*\Gamma\) by precomposing the marking \(m\) with the injection \(\phi\):
\[\phi^*(\Gamma) = (V,E,m\circ \phi).\]

It follows that
\[\chi(\phi^*\Gamma) = \chi(\Gamma).\]

If \(f:\Gamma_1 \to \Gamma_2\) is a morphism of \(r\)-marked graphs, then \(f\) induces a morphism \(\phi^*f: \phi^*\Gamma_1 \to \phi^*\Gamma_2\) in a natural way. It is straightforward to verify that \(\phi^*\) is a functor from \(\cG_r \to \cG_s\).

Similarly to the pushforward, it is easy to see that the pullback is well-behaved with respect to compositions.
\begin{prop} \label{prop:pullbackcomposition}
	Let \(\phi: \{1,\dots,s\} \to \{1,\dots,r\}\) and \(\psi: \{1,\dots,t\} \to \{1,\dots,s\}\) be injective maps. Then the functors \(\psi^*\phi^*\) and \((\phi\psi)^*\) from \(\cG_r\) to \(\cG_t\) are equal. \qed
\end{prop}

One can check that the pushforward and pullback functor are adjoints. Contrary to what the terms `pushforward' and `pullback' might suggest to a geometer, the pushforward functor is \emph{left} adjoint to the pullback. To ease our minds, we recall that the (left adjoint) pushforward functor does pushouts on sets of vertices, and the (right adjoint) pullback functor is a functor that \emph{forgets} some of the markings.

\section{Graphical formalism} \label{sec:tautformsassoctomarkedgraphs}

Fix an integer \(g\geq 2\). In this section, we will describe an operation that takes an \(r\)-marked graph and outputs a tautological form on \(\cC_g^r\). Let \(r, s \geq 0\) be a pair of integers and consider a map of sets \(\phi: \{1,\dots,s\}\to \{1,\dots,r\}\). We recall that to these data we have associated the tautological morphism
$ f^\phi: \cC_g^r \to \cC_g^s $
given by sending a marked family  \( (\cC \to S,\sigma_1,\dots,\sigma_r) \) of compact connected Riemann surfaces of genus $g$ to the marked family \( (\cC \to S,\sigma_{\phi(1)},\dots,\sigma_{\phi(s)}) \). A tautological morphism $f^\phi$ is a submersion if and only if the map $\phi$ is injective.

The following examples list some tautological morphisms that we often use. 

\begin{example}
	If \(1\leq i\leq r\) is an integer, the map \(\{1\} \to \{1,\dots,r\}\) given by \(1\mapsto i\) induces the map \(\cC_g^r \to \cC_g\) that projects onto the \(i\)th coordinate. We denote this map by \(p_i\). More generally, if \(1\leq i_1,\dots,i_s\leq r\) are integers, we denote by \[p_{i_1,\dots,i_s}: \cC_g^r\to\cC_g^s\] the tautological morphism associated to \(\phi: \{1,\dots,s\}\to \{1,\dots,r\}: k \mapsto i_k\). 
\end{example}
\begin{example}
	Let \(1\leq i_1<\dots<i_s\leq r\) be integers. Consider the unique increasing map \(\phi:\{1,\dots,r-s\} \to \{1,\dots,r\}\) whose image is \(\{1,\dots,r\} \setminus\{i_1,\dots,i_s\}\). Denote by
	\[ p_{(i_1,\dots,i_s)}: \cC_g^r \to \cC_g^{r-s} \]
	the tautological morphism associated to \(\phi\) (notice the parentheses!). Then \(p_{(i_1,\dots,i_s)}\) is the tautological submersion that `forgets the coordinates \(i_1,\dots,i_s\)'. For instance, the map \(p_{(2)}: \cC_g^2 \to \cC_g\) equals the map \(p_1:\cC_g^2\to\cC_g\). 
\end{example}

Consider a commutative diagram of sets, together with the associated diagram of moduli stacks:
\[
\begin{tikzcd}
	\{1,\dots,u\} & \{1,\dots,s\} \arrow[l,"\eta"'] \\
	\{1,\dots,t\} \arrow[u,"\chi"'] & \{1,\dots,r\} \arrow[l,"\psi"']\arrow[u,"\phi"']
\end{tikzcd}
\hspace{0.15\textwidth}
\begin{tikzcd}
	\cC_g^u \arrow[r,"f^\eta"]\arrow[d,"f^\chi"] & \cC_g^s \arrow[d,"f^\phi"] \\
	\cC_g^t \arrow[r,"f^\psi"] & \cC_g^r 
\end{tikzcd}
\]
It is not difficult to see that the diagram of moduli stacks is cartesian if and only if the diagram of sets is a pushout diagram. We will be using such cartesian diagrams often.

Let \(\Gamma = (V,E,m)\) be an \(r\)-marked graph, and let \(u\) be the number of unmarked vertices of \(\Gamma\). Choose a bijective extension \[\bar m: \{1,\dots,r+u\} \isoto V\] of the marking \(m: \{1,\dots,r\} \to V\). We will define a differential form \(\mu_\Gamma\) on \(\cC_g^{r+u}\) that will depend on the choice of this extension \(\bar m\). 

First, we associate to every edge \(e\in E\) a \(2\)-form \(h_e\) on \(\cC_g^{r+u}\). This form is defined as follows. Suppose that the endpoints of \(e\) are \(\bar m(i)\) and \(\bar m(j)\). We define
\[h_e = p_{i,j}^* h \in \tR^2(\cC_g^{r+u}),\]
where \(p_{i,j}: \cC_g^{r+u} \to \cC_g^2\) is the projection on the \(i\)th and \(j\)th coordinate. If $e$ is a loop based at vertex $\bar m(i)$, then
\[h_e = p_{i,i}^* h = p_i^* \Delta^* h = p_i^*e^A,\]
where $p_i: \cC_g^{r+u}\to \cC_g$ is the projection on the \(i\)th coordinate, and $\Delta: \cC_g \to \cC_g^2$ is the diagonal morphism. 
Notice that \(h_e\) does not depend on the order of \(i\) and \(j\) as the form \(h\) is symmetric in the two coordinates of \(\cC_g^2\). 

Now, we let \(\mu_\Gamma\) denote the product of all these 2-forms:
\[\mu_\Gamma = \bigwedge_{e\in E} h_e \in \tR^{2\abs{E}}(\cC_g^{r+u}).\]
This form depends on the choice of \(\bar m\). However, the form obtained from a different choice of \(\bar m\) only differs from \(\mu_\Gamma\) by permutation of the last \(u\) coordinates of \(\cC_g^{r+u}\). Therefore, by Fubini's theorem, the fiber integral
\begin{equation}\label{eqn:alpha_Gamma}
	\alpha_\Gamma := \int_{p_{1,\dots,r}: \cC_g^{r+u} \to \cC_g^r} \mu_\Gamma \in \tR^{2(\abs{E}-u)}(\cC_g^r)
\end{equation}
does not depend on the choice of \(\bar m\). It is clear that the forms $\alpha_\Gamma$ are all tautological forms. Note that the degree $2(|E|-u)$ of $\alpha_\Gamma$ can alternatively be written as $2(r-\chi(\Gamma))$, with $\chi(\Gamma)=r+u-|E|$ the Euler characteristic of $\Gamma$.

\begin{definition}
	Let \(\Gamma\) be an \(r\)-marked graph. The $2(r-\chi(\Gamma))$-form \(\alpha_\Gamma\)   on \(\cC_g^r\) defined in Equation~\ref{eqn:alpha_Gamma} is \emph{the (tautological) form associated to \(\Gamma\)}.  Here $\chi(\Gamma)$ denotes the Euler characteristic of $\Gamma$.
\end{definition}

As the following examples show, all tautological differential forms we found so far can be expressed as tautological forms associated to marked graphs. 
\begin{example} \label{ex:h-as-graph-form}
	Consider the unique \(2\)-marked graph \(\Gamma\) with no unmarked vertices and a single edge between the two marked vertices. The associated \(2\)-form \(\alpha_\Gamma\) on \(\cC_g^2\) is \(h\). 
	\begin{center}
		\begin{tikzpicture}
			\draw (0,0) node[left]{$\Gamma=$\;} node[mvertex,label=1]{} -- (1,0) node[mvertex,label=2]{};
		\end{tikzpicture}
	\end{center}
\end{example}
\begin{example} \label{ex:eA-as-graph-form}
	Consider the unique \(1\)-marked graph \(\Gamma\) with no unmarked vertices and a single loop based at the unique vertex of \(\Gamma\). The associated \(2\)-form \(\alpha_\Gamma\) on \(\cC_g\) is \(\Delta^* h = e^A\).
	\begin{center}
		\begin{tikzpicture}
			\draw (0,0) node[left]{$\Gamma=$\;} .. controls ++(45:1) and ++(-45:1) .. (0,0) node[mvertex,label=1]{};
		\end{tikzpicture}
	\end{center}
\end{example}
\begin{example} \label{ex:ddphi-as-graph-form}
	Consider the two 0-marked graphs in the following picture.
	
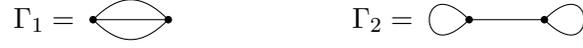
\begin{center}
		\begin{tikzpicture}
			\begin{scope}
				\draw (0,0) node[left]{\(\Gamma_1=\)\;};
				\draw (0,0) node[vertex]{} -- (1,0) node[vertex]{} ;
				\draw (0,0) .. controls ++(45:0.5) and ++(135:0.5) .. (1,0);
				\draw (0,0) .. controls ++(-45:0.5) and ++(-135:0.5) .. (1,0);
			\end{scope}
			\begin{scope}[shift={(5,0)}]
				\draw (-0.5,0) node[left]{\(\Gamma_2=\)\;};
				\draw (0,0) node[vertex]{} -- (1,0) node[vertex]{} ;
				\draw (0,0) .. controls ++(135:1) and ++(-135:1) .. (0,0);
				\draw (1,0) .. controls ++(45:1) and ++(-45:1) .. (1,0);
			\end{scope}
		\end{tikzpicture}
	\end{center}
	The associated forms on \(\cM_g\) are
	\[\alpha_{\Gamma_1} = \int_{\cC_g^2/\cM_g} h^3 = \nu \]
	and
	\begin{align*}
		\alpha_{\Gamma_2} &= \int_{\cC_g^2/\cM_g} h\wedge p_1^*e^A \wedge p_2^*e^A \\
				  &= \int_{\cC_g/\cM_g}\int_{p_1:\cC_g^2\to\cC_g} h\wedge p_1^*e^A \wedge p_2^*e^A\\
				  &= \int_{\cC_g/\cM_g} \left( e^A \wedge \int_{p_1} h\wedge p_2^*e^A \right)\\
				  &= \int_{\cC_g/\cM_g} (e^A)^2\\
				  &= e_1^A,
	\end{align*}
	where we have used the projection formula for fiber integrals and Lemma \ref{lem:inteAc1L}.
\end{example}

\section{Tautological forms and graph operations} \label{sec:taut_forms_and_graph_oper}

In this section we will see that the forms $\alpha_\Gamma$ based on marked graphs $\Gamma$ behave rather nicely with respect to pullbacks, pushforwards, and gluing  of marked graphs. By using this fact, we will be able to prove the following theorem. 
\begin{thm} \label{thm:taut-forms-generated-by-graphs}
	For every integer \(r\geq 0\), the ring of tautological differential forms \(\tR^*(\cC_g^r)\) is spanned as an \(\R\)-vector space by forms \(\alpha_\Gamma\) associated to \(r\)-marked graphs \(\Gamma\). 
\end{thm}

By our definition of tautological forms it suffices to prove that the system of linear subspaces \(S^*(\cC_g^r) \subseteq \tR^*(\cC_g^r)\) spanned by forms associated to \(r\)-marked graphs is a system of sub-\(\R\)-algebras, that the system is closed under pullbacks and fiber integrals, and that \(h\) is contained in \(S^*(\cC_g^2)\).  The last item is accomplished by Example \ref{ex:h-as-graph-form}. The first two items will be accomplished by the propositions below.

We start by proving that \(S^*(\cC_g^r) \subseteq \tR^*(\cC_g^r)\) is a subring for every \(r\geq 0\). First of all, the form associated to the unique \(r\)-marked graph consisting of \(r\) vertices and no edges is \(1\). The following proposition implies that \(S^*(\cC_g^r)\) is closed under wedge products and therefore a subring of \(\tR^*(\cC_g^r)\). 

\begin{prop} \label{prop:graphformwedge}
	Let \(\Gamma = (V,E,m)\) and \(\Gamma' = (V',E',m')\) be two \(r\)-marked graphs, and let \(\alpha_\Gamma\) and \(\alpha_{\Gamma'}\) be the associated tautological forms on \(\cC_g^r\). Then
	\[\alpha_\Gamma \wedge \alpha_{\Gamma'} = \alpha_{\Gamma \sqcup_r \Gamma'}.\]
\end{prop}
\begin{proof}
	Assume that \(\Gamma\) and \(\Gamma'\) have respectively \(u\) and \(u'\) unmarked vertices. Choose bijective extensions
	\begin{align*}
		\bar m: \{1,\dots,r+u\} & \isoto V \\
		\bar m': \{1,\dots,r+u'\} & \isoto V'
	\end{align*}
	of \(m\) and \(m'\). Let \(\phi: \{1,\dots,r+u\} \to \{1,\dots,r+u+u'\}\) be the inclusion, and define the map
	\[
		\psi: \{1,\dots,r+u'\} \to \{1,\dots,r+u+u'\}: k \mapsto \begin{cases} k & \text{if } k\leq r\\ k+u & \text{if } k>r. \end{cases}
	\]
	It follows that the diagram
	\[
	\begin{tikzcd}
		\{1,\dots,r+u+u'\} & \{1,\dots,r+u\} \arrow[l,"\phi"'] \\
		\{1,\dots,r+u'\} \arrow[u,"\psi"'] & \{1,\dots,r\} \arrow[u, hook]\arrow[l,hook']
	\end{tikzcd}
	\]
	is a pushout diagram of sets, so we have the associated cartesian diagram of moduli stacks
	\[ 
	\begin{tikzcd}[column sep=7em]
		\cC_g^{r+u+u'} \arrow[r,"p_{1,\dots,r+u}"] \arrow[d,"p_{1,\dots,r,r+u+1,\dots,r+u+u'}"'] & \cC_g^{r+u} \arrow[d,"p_{1,\dots,r}"] \\ 
		\cC_g^{r+u'} \arrow[r,"p_{1,\dots,r}"] & \cC_g^r.
	\end{tikzcd}
	\]
	Now let \(\Gamma'' = (V'',E'',m'') = \Gamma \sqcup_r \Gamma'\). By the universal property of the pushout, we have an induced \(r+u+u'\)-marking
	\[\bar m'': \{1,\dots,r+u+u'\} \isoto V''\]
	of the set of vertices \(V''\) of \(\Gamma''\) that extends \(m''\). If \(e\in E\) is an edge in \(\Gamma\) between vertices \(\bar m(i)\) and \(\bar m(j)\), then the corresponding edge in \(\Gamma''\) has endpoints \(\bar m''(\phi(i))\) and \(\bar m''(\phi(j))\). Similarly, if \(e\in E'\) is an edge in \(\Gamma'\) between vertices \(\bar m'(i)\) and \(\bar m'(j)\), then the corresponding edge in \(\Gamma''\) has endpoints \(\bar m''(\psi(i))\) and \(\bar m''(\psi(j))\). It follows that
	\begin{align*}
		\mu_{\Gamma''} &= \bigwedge_{e\in E''} h_e \\
					   &= \bigwedge_{e\in E} p_{1,\dots,r+u}^* h_e \wedge \bigwedge_{e\in E'} p_{1,\dots,r,r+u+1,\dots,r+u+u'}^* h_e \\
		      &= p_{1,\dots,r+u}^* \mu_\Gamma \wedge p_{1,\dots,r,r+u+1,\dots,r+u+u'}^* \mu_{\Gamma'}.
	\end{align*}
	Using the base change formula and the projection formula for fiber integrals, we find that the fiber integral \(\alpha_{\Gamma \sqcup_r \Gamma'}\) equals \(\alpha_\Gamma \wedge \alpha_{\Gamma'}\). 	
\end{proof}

Next, we will show that the system of vector spaces \(S^*(\cC_g^r)\subseteq \tR^*(\cC_g^r)\) is closed under pullbacks along tautological morphisms. Let \(f^\phi: \cC_g^r \to \cC_g^s\) be a tautological morphism, induced by a map \(\phi: \{1,\dots,s\} \to \{1,\dots,r\}\). Recall  that \(\phi\) induces a pushforward operator \(\phi_*: \cG_r \to \cG_s\) from \(r\)-marked graphs to \(s\)-marked graphs. The following proposition implies that the pullback map \(f^{\phi,*}\) on differential forms is compatible with the pushforward map on graphs. From this one easily deduces that the system of forms \(S^*(\cC_g^r)\) is closed under pullbacks along tautological maps.
\begin{prop} \label{prop:graphformpullback}
	Let \(f^\phi: \cC_g^r \to \cC_g^s\) be the tautological morphism associated to a map \(\phi: \{1,\dots,s\}\to\{1,\dots,r\}\). Suppose that \(\alpha_\Gamma \in S^*(\cC_g^s)\) is the form associated to an \(s\)-marked graph \(\Gamma\). Then 
	\[f^{\phi,*} \alpha_\Gamma = \alpha_{\phi_*\Gamma}\]
	with \(\phi_*\Gamma\) the pushforward of \(\Gamma\) along \(\phi\).
\end{prop}
\begin{proof}
	The proof is similar to the proof of Proposition \ref{prop:graphformwedge}, so only a short sketch is given here. We extend the labeling on \(\Gamma\) to an \((s+u)\)-labeling, with \(u\) the number of unmarked vertices of \(\Gamma\). This induces an \((r+u)\)-labeling of \(\phi_*\Gamma\), and it follows that the pullback of \(\mu_\Gamma\) along the induced map \(\cC_g^{r+u} \to \cC_g^{s+u}\) equals \(\mu_{\phi_*\Gamma}\). By the base change formula the desired result follows. 
\end{proof}

Now, let \(f^\phi: \cC_g^r \to \cC_g^s\) be a tautological submersion, associated to an injective map \(\phi: \{1,\dots,s\} \hookrightarrow \{1,\dots,r\}\).	Recall that we have a pullback map \(\phi^*: \cG_r \to \cG_s\). The following proposition shows that the pullback map on graphs is compatible with the fiber integral map on differential forms. This implies that the system \(S^*(\cC_g^r)\subseteq \tR^*(\cC_g^r)\) is closed under fiber integrals.

\begin{prop} \label{prop:graphformintegral}
	Let \(\phi: \{1,\dots,s\} \to \{1,\dots,r\}\) be an injective map, and let \(f^\phi: \cC_g^r \to \cC_g^s\) be the associated tautological submersion. Let \(\Gamma\in\cG_r\) be an \(r\)-marked graph, and let \(\phi^*\Gamma\) be the \(s\)-marked graph induced by \(\phi\). Then
	\[\int_{f^\phi} \alpha_\Gamma = \alpha_{\phi^*\Gamma}\]
\end{prop}
\begin{proof}
	Let \(u\) be the number of unmarked vertices in \(\Gamma\). Extend the inclusion \(\phi: \{1,\dots,s\} \to \{1,\dots,r\}\) to a permutation \(\{1,\dots,r\} \to \{1,\dots,r\}\), and then join this map with the identity on \(\{r+1,\dots,r+u\}\) to obtain a bijective map
	\[\bar\phi: \{1,\dots,r+u\} \isoto \{1,\dots,r+u\}\]
	that extends \(\phi\).

	Moreover, choose a bijective extension \(\bar m:\{1,\dots,r+u\} \isoto V\) of the marking \(m\) of \(\Gamma\). We immediately obtain an extension \[\overline{m\phi} = \bar m \circ \bar \phi: \{1,\dots,r+u\} \isoto V\] of the marking \(m\phi\) of the \(s\)-marked graph \(\phi^*\Gamma = (V,E,m\phi)\). We have a commutative diagram of sets, inducing a commutative diagram of moduli stacks:
	\[ \begin{tikzcd}
		\{1,\dots,r+u\} & \{1,\dots,r\} \arrow[l,"\supseteq"'] \\
		\{1,\dots,r+u\} \arrow[u,"\bar\phi"'] & \{1,\dots,s\} \arrow[l,"\supseteq"'] \arrow[u,"\phi"']
	\end{tikzcd} \hspace{0.15\linewidth}
	\begin{tikzcd}[column sep=7ex]
		\cC_g^{r+u} \arrow[r,"p_{1,\dots,r}"]\arrow[d,"f^{\bar\phi}"] & \cC_g^r \arrow[d,"f^\phi"] \\
		\cC_g^{r+u} \arrow[r,"p_{1,\dots,s}"] & \cC_g^s.
	\end{tikzcd}
	\]

	If \(e\) is an edge in \(\Gamma\) with endpoints \(\bar m(i), \bar m(j)\), then the corresponding edge \(\phi^* e\) in \(\phi^*\Gamma\) has endpoints \(\overline{m\phi}(\bar\phi^{-1}(i))\) and \(\overline{m\phi}(\bar\phi^{-1}(j))\). It follows that the corresponding 2-forms on \(\cC_g^{r+u}\) are related as follows:
	\[h_e = f^{\bar\phi,*} h_{\phi^* e}.\]
	From this, we find that 
	\[\mu_\Gamma = f^{\bar\phi,*} \mu_{\phi^*\Gamma},\]
	so
	\[\mu_{\phi^*\Gamma} = \int_{f^{\bar\phi}} \mu_\Gamma.\]
	We therefore have:
	\[\int_{f^\phi} \alpha_\Gamma = \int_{f^\phi}\int_{p_{1,\dots,r}} \mu_\Gamma = \int_{p_{1,\dots,s}}\int_{f^{\bar\phi}} \mu_\Gamma = \int_{p_{1,\dots,s}} \mu_{\phi^*\Gamma} = \alpha_{\phi^*\Gamma},\]
	proving the proposition.
\end{proof}

\section{Contracted graphs}\label{sec:contractedgraphs}

Let \(\Gamma\) be an \(r\)-marked graph. 
\begin{definition}
	We say \(\Gamma\) is \emph{contracted} if all its unmarked vertices have degree at least 3, and each unmarked vertex of degree 3 is incident to three distinct edges. 
\end{definition}

If a graph is not contracted, we can attempt to turn this graph into a contracted graph by altering the problematic vertices. 

\begin{definition} \label{def:contracting}
	Let \(\Gamma=(V,E,m)\) be an \(r\)-marked graph, and let \(v\in V\) be an unmarked vertex such that \(\deg(v)\leq 2\), or such that \(\deg(v)=3\) and \(v\) is incident to a loop. The graph obtained from \(\Gamma\) by \emph{contracting} \(v\) is an \(r\)-marked graph \(\Gamma'\) defined by the following operation:
	\begin{enumerate}\setcounter{enumi}{-1}
		\item If \(\deg v = 0\), remove \(v\);\label{def:contr-0}
		\item If \(\deg v = 1\), remove \(v\) and the unique edge incident to \(v\);\label{def:contr-1}
		\item If \(\deg v = 2\), \emph{smooth out} the vertex \(v\); that is:
			\begin{enumerate}
				\item If \(v\) is incident to two distinct edges, whose other endpoints \(w,w'\) are distinct, remove \(v\) and these two edges, and add an edge between \(w\) and \(w'\);\label{def:contr-2a}
				\item If \(v\) is incident to two distinct edges, whose second endpoint is the same vertex \(w\), remove \(v\) and these two edges, and add a loop at \(w\);\label{def:contr-2b}
				\item If \(v\) is incident to a single loop, remove \(v\) and this loop;\label{def:contr-2c}
			\end{enumerate}
		\item If \(\deg v = 3\), and \(w\) is the other endpoint of the non-loop edge incident to \(v\), remove \(v\), this edge, and the loop at \(v\), and add a loop at \(w\).\label{def:contr-3}
	\end{enumerate}
\end{definition}

It follows that the vertex set of the graph obtained from \(\Gamma\) by contracting \(v\) is equal to \(V\setminus \{v\}\). In case (0), the Euler characteristic drops by $1$ upon contracting; in the other cases (1)--(3) the Euler characteristic remains the same.

If we are given an \(r\)-marked graph \(\Gamma\), we can always reduce \(\Gamma\) to a contracted \(r\)-marked graph by applying a finite amount of graph contractions. 
As the contraction operations only apply to unmarked vertices, it follows that  contraction of vertices commutes with gluing of \(r\)-marked graphs.

\subsection{Counting contracted graphs}\label{sec:countingcontractedgraphs}
The following result will be crucial in obtaining our basic finiteness result on tautological rings.

\begin{thm} \label{thm:finmanycontrgraphs}
	Let \(r\geq 0\), and \(\chi\in\Z\). There are, up to isomorphism, only finitely many contracted \(r\)-marked graphs of characteristic \(\chi\).
\end{thm}
\begin{proof}
	Let \(\Gamma\) be a contracted \(r\)-marked graph of characteristic \(\chi\), and let \(u\) denote its number of unmarked vertices, and \(e\) its number of edges. As every unmarked vertex has degree at least 3 it follows that
	\[2e = \sum_{v\in \Gamma} \deg(v) \geq 3u.\]
	After substituting \(e=r+u-\chi\), we find:
	\[u\leq 2r - 2\chi,\]
	and hence
	\[e = r+u-\chi \leq 3r-3\chi.\]
	We have obtained upper bounds for the number of vertices and edges of \(\Gamma\), and a simple combinatorial argument then shows that there can only be finitely many graphs of this form up to isomorphism. 
\end{proof}

\section{Tautological forms and contractions}
 \label{sec:taut_forms_contr}

The purpose of this section is to show that contracted graphs suffice to span the rings of tautological forms. Let \(g\geq 2\) be an integer.
\begin{thm} \label{lem:spanned-by-contracted-graphs}
	Let \(d\geq 0\) and \(r\geq 0\) be integers. The space \(\tR^{2d}(\cC_g^r)\) of tautological forms of degree \(2d\) on \(\cC_g^r\) is the linear span of the forms \(\alpha_\Gamma\) associated to contracted \(r\)-marked graphs \(\Gamma\) with Euler characteristic \(\chi(\Gamma) = r-d\).  
\end{thm}
By combining Theorem \ref{lem:spanned-by-contracted-graphs} with Theorem \ref{thm:finmanycontrgraphs}, we obtain the following.
\begin{thm} \label{thm:taut-ring-degreewise-finite}
	 For all integers \(r\geq 0\) and \(d\geq 0\), the space \(\tR^{2d}(\cC_g^r)\) of tautological forms of degree \(2d\) on \(\cC_g^r\) is finite-dimensional. \qed
\end{thm}
For the proof of Theorem~\ref{lem:spanned-by-contracted-graphs} we need the following technical result.
\begin{prop} \label{prop:contractions}
	Let \(\Gamma=(V,E,m)\) be an \(r\)-marked graph, and suppose that \(\Gamma\) has an unmarked vertex \(v\), such that either $\deg(v)\leq 2$, or $\deg(v)=3$ and $v$ is incident to precisely two edges. Let $\Gamma'$ be the $r$-marked graph obtained from $\Gamma$ by contracting $v$, see Definition~\ref{def:contracting}. 
	\begin{enumerate}
		\item[0.] \label{prop:contr-deg0}If \(\deg v = 0\), then \(\alpha_\Gamma = 0\).  
		\item[1.] \label{prop:contr-deg1}If \(\deg v = 1\), then \(\alpha_\Gamma = \alpha_{\Gamma'}\).
		\item[2a.] \label{prop:contr-deg2-edge}Suppose that \(\deg v = 2\) and that \(v\) has two distinct neighbors \(w\neq w'\). Then \(\alpha_\Gamma = \alpha_{\Gamma'}\).
		\item[2b.] \label{prop:contr-deg2-loop}Suppose that \(\deg v = 2\) and that \(v\) has a single neighbor \(w\neq v\). Then \(\alpha_{\Gamma} = \alpha_{\Gamma'}\).
		\item[2c.] \label{prop:contr-deg2-component}Suppose that \(\deg v = 2\) and that \(v\) is its own neighbor; that is: there is a loop based at \(v\). Then \(\alpha_\Gamma = (2-2g)\alpha_{\Gamma'}\). 
		\item[3.] \label{prop:contr-deg3-loop}Suppose that \(\deg v = 3\) and that $v$ is incident to precisely two edges. Then $\alpha_\Gamma = \alpha_{\Gamma'}$. 
	\end{enumerate}
\end{prop}
\begin{proof}[Proof of Theorem~\ref{lem:spanned-by-contracted-graphs}] 
Theorem~\ref{thm:taut-forms-generated-by-graphs}  implies that the space \(\tR^{2d}(\cC_g^r)\)  is the linear span of the forms associated to  \(r\)-marked graphs \(\Gamma\) with Euler characteristic \(\chi(\Gamma) = r-d\).
Proposition \ref{prop:contractions} implies that if $\Gamma$ is an  \(r\)-marked graph and $\Gamma'$ is  an  \(r\)-marked graph that is obtained from $\Gamma$ by successively contracting vertices, then the form $\alpha_\Gamma$ can be obtained from the form $\alpha_{\Gamma'}$ by multiplying it by zero or a power of \((2-2g)\). This shows that \(\tR^{2d}(\cC_g^r)\) is the linear span of forms associated to contracted \(r\)-marked graphs. 
\end{proof}
\begin{proof}[Proof of Proposition \ref{prop:contractions}]
	Let \(\Gamma = (V,E,m)\) be an \(r\)-marked graph, and let \(v\) be an unmarked vertex of degree \(\leq 2\), or an unmarked vertex of degree \(\leq 3\) with a loop. 
	Define a graph \(\Gamma''\) by removing \(v\), and all edges emanating from \(v\), from \(\Gamma\). Moreover, we have the graph $\Gamma'$ that is obtained from $\Gamma$ by contracting $v$. 

	The graph \(\Gamma''\) represents an `intermediate step' in obtaining \(\Gamma'\) from \(\Gamma\). The following picture describes the situation in the case where \(v\) has two distinct neighbors.
	
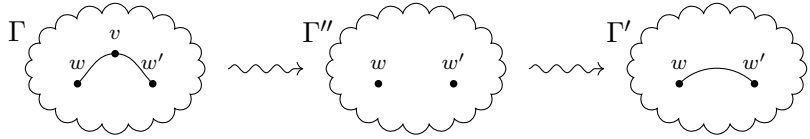
\begin{center}
		\begin{tikzpicture}
			\begin{scope}
				\node [cloud, draw, cloud puffs=20, cloud puff arc=150, aspect=1.5, inner ysep=15] {};
				\draw (-0.5,-0.2) node[vertex,label={\scriptsize\(w\)}]{} to[out=45,in=180] (0,0.2) node[vertex,label={\scriptsize\(v\)}]{} to[out=0,in=135] (0.5,-0.2) node[vertex,label={\scriptsize\(w'\)}]{} ;
				\draw (-1.3,0.5) node {\(\Gamma\)};
			\end{scope}
			\begin{scope}[shift={(4,0)}]
				\node [cloud, draw, cloud puffs=20, cloud puff arc=150, aspect=1.5, inner ysep=15] {};
				\draw (-0.5,-0.2) node[vertex,label={\scriptsize\(w\)}]{} (0.5,-0.2) node[vertex,label={\scriptsize\(w'\)}]{} ;
				\draw (-1.3,0.5) node {\(\Gamma''\)};
			\end{scope}
			\begin{scope}[shift={(8,0)}]
				\node [cloud, draw, cloud puffs=20, cloud puff arc=150, aspect=1.5, inner ysep=15] {};
				\draw (-0.5,-0.2) node[vertex,label={\scriptsize\(w\)}]{} to[out=45,in=135] (0.5,-0.2) node[vertex,label={\scriptsize\(w'\)}]{} ;
				\draw (-1.3,0.5) node {\(\Gamma'\)};
			\end{scope}
			\draw [->,decorate,decoration={snake,amplitude=1.5,segment length=8}] (1.5,0) -- (2.5,0);
			\draw [->,decorate,decoration={snake,amplitude=1.5,segment length=8}] (5.5,0) -- (6.5,0);
		\end{tikzpicture}
	\end{center}

	Let \(u\geq 0\) be such that \(\Gamma\) has \(u+1\) unmarked points. Fix an extension of \(m\) to an \((r+u+1)\)-marking
	\[\bar m: \{1,\dots,r+u+1\} \isoto V,\]
	such that $\bar m(r+u+1)=v$. 

	Restricting \(\bar m\) to \(\{1,\dots,r+u\}\) induces an \((r+u)\)-marking on \(\Gamma'\) and \(\Gamma''\) that extends the \(r\)-marking on these graphs. We obtain differential forms \(\mu_\Gamma\), \(\mu_{\Gamma'}\), and \(\mu_{\Gamma''}\) that live on \(\cC_g^{r+u+1}\), \(\cC_g^{r+u}\), and \(\cC_g^{r+u}\), respectively.

	The inclusions \(\{1,\dots,r\} \subseteq \{1,\dots,r+u\} \subseteq \{1,\dots,r+u+1\}\) induce tautological submersions 
	\[
		\begin{tikzcd}
			\cC_g^{r+u+1} \arrow[r,"q"] \arrow[dr,"pq"'] & \cC_g^{r+u} \arrow[d,"p"] \\
								    & \cC_g^r.
		\end{tikzcd}
	\]
	We have
	\[\alpha_\Gamma = \int_{pq} \mu_\Gamma \quad\text{and}\quad \alpha_{\Gamma'} = \int_{p} \mu_{\Gamma'}.  \]
	If we can prove that \(\smallint_q \mu_\Gamma = 0\) in case 0, \(\smallint_q \mu_\Gamma = \mu_{\Gamma'}\) in cases 1, 2a, 2b, and 3, and \(\smallint_q \mu_\Gamma = (2-2g)\mu_{\Gamma'}\) in case 2c, we are done. 
	\begin{enumerate}
		\item[0.] Suppose \(v\) has degree 0. The set of edges of \(\Gamma\) is equal to the set of edges of \(\Gamma'\), so we obtain
			\[\mu_\Gamma = q^*\mu_{\Gamma'}.\]
			Taking fiber integrals and applying the projection formula yields:
			\[\int_q \mu_\Gamma = \mu_{\Gamma'} \int_q 1 = 0,\]
			and we find that \(\alpha_\Gamma = 0\).
		\item[1.] Suppose \(v\) has degree \(1\); let \(i\in\{1,\dots,r+u\}\) be such that \(\bar{m}(i)\) is the neighbor of \(v\). The graph \(\Gamma\) is obtained from \(\Gamma'\) by adding the vertex \(v\) and the edge between \(v\) and \(\bar{m}(i)\). We therefore have:
			\[\mu_\Gamma = q^* \mu_{\Gamma'} \wedge p_{i,r+u+1}^* h,\]
			so
			\[\int_q \mu_\Gamma = \mu_{\Gamma'} \wedge \int_q p_{i,r+u+1}^* h.\]
			By using the base change formula with the cartesian diagram
			\[\begin{tikzcd}[column sep=8ex]
					\cC_g^{r+u+1} \arrow[rd,phantom,"\square"] \arrow[r,"p_{i,r+u+1}"] \arrow[d,"q"] & \cC_g^2 \arrow[d,"p_1"] \\
				\cC_g^{r+u} \arrow[r,"p_i"'] & \cC_g,
			\end{tikzcd}\]
			we find:
			\[\int_q p_{i,r+u+1}^* h = p_i^* \int_{p_1} h = 1,\]
			where the latter equality follows from Equation~\ref{lem:inteA}.
			This shows that \(\smallint_q \mu_\Gamma = \mu_{\Gamma'}\), so \(\alpha_\Gamma = \alpha_{\Gamma'}\).
		\item[2a.] Suppose \(v\) has degree \(2\), and that \(v\) has two distinct neighbors \(w\) and \(w'\). Let \(i,j\in \{1,\dots,r+u\}\) be such that \(\bar m(i) = w\) and \(\bar m(j)=w'\). In this case, we find
			\[\mu_\Gamma = q^*\mu_{\Gamma''} \wedge p_{i,r+u+1}^* h \wedge p_{j,r+u+1}^* h,\]
			and
			\[\mu_{\Gamma'} = \mu_{\Gamma''} \wedge p_{i,j}^* h.\]
			In this case, another application of the base change formula, together with the identity of forms 
			\[\int_{p_{12}} p_{13}^* h \wedge p_{23}^* h = h\]
			from Lemma \ref{lem:p13hp23h}
			shows that \(\int_q \mu_\Gamma = \mu_{\Gamma'}\), and hence \(\alpha_\Gamma = \alpha_{\Gamma'}\).
		\item[2b.] The proof in this case is very similar to the proofs for cases 1 and 2a. In this case, we use the identity
			\[\int_{p_1} h^2 = e^A\]
			from Lemma \ref{lem:inteAc1L}.
		\item[2c.] Again, the proof of this case is similar to that of the previous cases. The identity used here is
			\[\int_{\cC_g/\cM_g} e^A = (2-2g),\]
			see Equation~\ref{lem:inteA}.
		\item[3.] Finally, the proof in case 3 is analogous to that of earlier cases, where we use the identity
			\[\int_{p_1} h \wedge p_2^* e^A = e^A\]
			from Lemma \ref{lem:inteAc1L}.
	\end{enumerate}
\end{proof}

\section{Proof of Theorem~\ref{thm:finite_dim}} \label{sec:proof_fin_dim} 

Theorem~\ref{thm:finite_dim} is a consequence of Theorem~\ref{thm:taut-ring-degreewise-finite} as follows. 

Let $\cT_g$ denote Teichm\"uller space in genus $g$. Let   \(\cX_g\to \cT_g\) denote the universal family of genus \(g\) Riemann surfaces with Teichm\"uller structure and \(\cX_g^r\) the \(r\)-fold fiber product of \(\cX_g\) over \(\cT_g\). We have a natural surjective submersion $\cX_g^r \to \cC_g^r$. By Lemma~\ref{lem:submersion_inj} we obtain an inclusion \(A^*(\cC_g^r) \to A^*(\cX_g^r)\). As \(\cX_g^r\) is a manifold of (real) dimension \(6g-6+2r\), it follows that \(A^d(\cX_g^r)\) is zero for all \(d>6g-6+2r\). We see that the same is true for \(A^d(\cC_g^r)\) and hence for \(\tR^d(\cC_g^r)\). We see that \[\tR^*(\cC_g^r) = \bigoplus_{d\geq 0} \tR^d(\cC_g^r) = \bigoplus_{d=0}^{3g-3+r} \tR^{2d}(\cC_g^r).\] 
Each of the finitely many direct summands is finite-dimensional by Theorem~\ref{thm:taut-ring-degreewise-finite}, hence the ring \(\tR^*(\cC_g^r)\)
 is itself finite-dimensional. This proves Theorem~\ref{thm:finite_dim}.

\begin{remark}  For each $d \in \Z_{\ge 0}$ there exists a polynomial $f_d \in \Q[X]$ of degree~$2d$ such that for all $r \in \Z_{\ge 0}$ the number of isomorphism classes of contracted \(r\)-marked graphs of characteristic \(r-d\) is given by $f_d(r)$. We refer to \cite[Section~3.8]{thesis} for a proof of this statement. We see that for $d$ fixed, the dimension $\dim \tR^{2d}(\cC_g^r) $ is bounded by a polynomial in $r$, independent of the genus.
\end{remark}

\section{Proof of Theorems~\ref{thm:degree_two} and~\ref{thm:ZK_inv}}\label{sec:2forms}

Let \(r\geq 0\) be an integer. We start with the following result.
\begin{thm} \label{thm:generate_deg_two}
The degree-two part \(\tR^2(\cC_g^r)\) of the ring of tautological forms is spanned by the following collection of 2-forms:
\begin{equation}\label{eqn:formsspanningR2}
	\{p_{ij}^* h: 1\leq i<j\leq r\} \cup \{p_i^*e^A: 1\leq i\leq r\} \cup \{e_1^A, \nu\}.
\end{equation}
\end{thm}
\begin{proof}
By Theorem \ref{thm:taut-ring-degreewise-finite}, we find that the space  \(\tR^2(\cC_g^r)\) is spanned by forms \(\alpha_\Gamma\), where \(\Gamma\) ranges over all contracted \(r\)-marked graphs of characteristic \(r-1\). It is not hard to show using an argument as in the proof of Theorem~\ref{thm:finmanycontrgraphs} that a contracted \(r\)-marked graph of characteristic \(r-1\) is one of the following graphs:
\begin{itemize}
	\item Graphs \(\Gamma\) with \(r\) marked vertices, no unmarked vertices, and a single edge: 
	\begin{center}
				\begin{tikzpicture}
					\begin{scope}
						\draw (0,0) node[mvertex,label=1]{} -- (.7,0) node[mvertex,label=2]{};
					\end{scope}
					\begin{scope}[shift={(3,0)}]
						\draw (0,0) node[mvertex,label=1]{} .. controls ++(135:.5) and ++(-135:.5) .. (0,0)  (.7,0) node[mvertex,label=2]{};
					\end{scope}
					\begin{scope}[shift={(6,0)}]
						\draw (0,0) node[mvertex,label=1]{}  (.7,0) node[mvertex,label=2]{} .. controls ++(45:.5) and ++(-45:.5) .. (.7,0);
					\end{scope}
				\end{tikzpicture}
			\end{center}
	If this edge is a loop based at vertex \(i\) then the associated form is \[\alpha_\Gamma = p_i^*e^A.\] If the edge is not a loop, and its endpoints are vertices \(i\) and \(j\), then the associated form is \[\alpha_\Gamma = p_{ij}^*h.\]
	\item The graph \(\Gamma\) with \(r\) marked vertices, one unmarked vertex, and two loops based at the unmarked vertex: 
	\begin{center}
				\begin{tikzpicture}
					\draw (0,0) node[mvertex,label=1]{} (.7,0) node[mvertex,label=2]{} (1.4,0) node[vertex]{}  .. controls ++(135:.5) and ++(-135:.5) .. ++(0,0) .. controls ++(45:.5) and ++(-45:.5) .. ++(0,0);
				\end{tikzpicture}
			\end{center}
	The associated form is
		\[\alpha_\Gamma = \int_{p_{1,\dots,r}:\cC_g^{r+1} \to \cC_g^r} p_{r+1}^*(e^A)^2 = e_1^A\]
		by the base change formula. Note the slight abuse of notation here: we write \(e_1^A\) for the pullback of \(e_1^A\) along the tautological morphism \(\cC_g^r\to\cM_g\). 
	\item The graph \(\Gamma\) with \(r\) marked vertices, two unmarked vertices, and three edges between the unmarked vertices: 
	
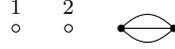
\begin{center}
				\begin{tikzpicture}
					\draw (0,0) node[mvertex,label=1]{} (.7,0) node[mvertex,label=2]{} (1.4,0) node[vertex]{}  .. controls ++(45:.35) and ++(135:.35) .. (2.1,0) node[vertex]{} .. controls ++(-135:.35) and ++(-45:.35) .. (1.4,0) -- (2.1,0);
				\end{tikzpicture}
			\end{center}
	By using the base change formula we obtain 
		\[\alpha_\Gamma = \int_{p_{1,\dots,r}: \cC_g^{r+2} \to \cC_g^r} p_{r+1,r+2}^* h^3 = \nu\]
		where we again abuse the notation by writing \(\nu\) for the pullback of \(\nu\) along the projection \(\cC_g^r\to \cM_g\).
\end{itemize}
The theorem follows.
\end{proof}
\begin{lem} \label{prop:exist_unique_potential} Assume that $g \ge 3$. Every pluriharmonic function on $\cM_g$ is constant.
\end{lem}
\begin{proof} This follows from the existence of a Satake compactification of $\cM_g$, where all boundary components have complex codimension at least two. See also \cite[Lemma 8.1]{kawazumi2009canonical}.
\end{proof}
Let $\varphi \colon \cM_g \to \R$ be the Kawazumi--Zhang invariant, see Equation~\ref{eqn:def_ZK_inv}.
\begin{lem} \label{prop:non-zero} The tautological $2$-form \( \frac{1}{\pi\i} \del\delbar\varphi  \) on $\cM_g$ is non-zero. 
\end{lem}
\begin{proof} By observing the asymptotic behavior of \(\varphi\) near the boundary of \(\cM_g\) studied in \cite{dejong2014}, we find  that \(\varphi\) is not constant.  When $g \ge 3$ the result then follows from Lemma~\ref{prop:exist_unique_potential}. When $g=2$, the result follows from the identity $(\Delta-5)\varphi=0$, proved in \cite{Matching}, where $\Delta$ is the Laplace-Beltrami operator with respect to the metric on $\cM_2$ induced from the Siegel metric on the Siegel upper half space in degree two.
\end{proof}
\begin{lem}  \label{prop:some_not_exact} Assume that $g \ge 3$. The form \(e_1^A\) is not exact, in particular \(e_1^A\) and \(\frac{1}{\pi\i} \del\delbar\varphi\) are linearly independent elements of \(\tR^2(\cM_g)\).
\end{lem}
\begin{proof} It follows from \cite[Theorem]{Mumford_Ab_Quot} that the Picard group $\Pic(\cM_g)$ injects into the cohomology group $H^2(\cM_g,\Q)$. Let $\lambda_1 \in \Pic(\cM_g)$ be the first Chern class of the Hodge bundle on $\cM_g$. It is shown in \cite[Section~5]{Mumford1983} that $\kappa_1 = 12\lambda_1$, and in \cite[Theorem~1]{arbarellocornalba1987} it is proved that $\lambda_1$ freely generates the Picard group of $\cM_g$. It follows that the cohomology class $\kappa_1$ does not vanish. This proves that $e_1^A$ is not exact. The linear independence then follows from Lemma~\ref{prop:non-zero}.
\end{proof}

\begin{prop} \label{prop:dimR2Mg}
	If \(g=2\), then \(\tR^2(\cM_g)\) is one-dimensional, and spanned by \(e_1^A\). If \(g\geq 3\), then \(\tR^2(\cM_g)\) is two-dimensional, and spanned by \(e_1^A\) and \(\nu\).
\end{prop}
\begin{proof} It follows from Theorem~\ref{thm:generate_deg_two} that  \(\tR^2(\cM_g)\) is spanned by \(\nu\) and \(e_1^A\). Therefore, the dimension of \(\tR^2(\cM_g)\) is at most two. From Lemma~\ref{prop:non-zero} combined with Equation \ref{eqn:lin_comb} we find that the dimension of \(\tR^2(\cM_g)\) is at least one. We obtain the proposition for \(g=2\) by observing that in \( \tR^2(\cM_2) \) we have Equation \ref{eq:special_rel_g=2}.
	We obtain the proposition for  \(g\geq 3\)  by Lemma \ref{prop:some_not_exact}.
\end{proof}

\begin{proof}[Proof of Theorem \ref{thm:degree_two}]
We use induction on~$r$.
	The case \(r=0\) is proved in Proposition \ref{prop:dimR2Mg}. For the case \(r=1\) we observe that  by Lemma~\ref{lem:submersion_inj}  the projection \(p:\cC_g \to \cM_g\) induces an inclusion \(p^* : \tR^2(\cM_g) \to \tR^2(\cC_g)\). Moreover, forms in the image of \(p^*\) are in the kernel of the fiber integral along \(p\),  by the projection formula.
	As 
	\[\int_p e^A = (2-2g) \neq 0,\]
	we find that \(e^A\) is not an element of \(p^* \tR^2(\cM_g)\). As $\tR^2(\cC_g)$ is spanned by the forms $e^A$, $e_1^A$, and $\nu$, see Theorem~\ref{thm:generate_deg_two}, we obtain the statement in the case $r=1$.

	Now let \(r\geq 2\), and assume that $\tR^2(\cC_g^s)$ has a basis as given in the Theorem for $0\le s <r$.
	Consider the following three tautological morphisms:
	\begin{align*}
		p_{(r)}: \cC_g^r \to \cC_g^{r-1}:&\; (x_1,\dots,x_r) \mapsto (x_1,\dots,x_{r-1});\\
		p_{(r-1)}: \cC_g^r \to \cC_g^{r-1}:&\; (x_1,\dots,x_r) \mapsto (x_1,\dots,x_{r-2},x_r);\\
		q_{(r-1)}: \cC_g^{r-1} \to \cC_g^{r-2}:&\; (x_1,\dots,x_{r-1}) \mapsto (x_1,\dots,x_{r-2}).
	\end{align*}
	We have a cartesian square
	\[
	\begin{tikzcd}
		\cC_g^r \arrow[r,"p_{(r)}"]\arrow[d,"p_{(r-1)}"']\drar[phantom,"\square"] & \cC_g^{r-1} \arrow[d,"q_{(r-1)}"] \\
		\cC_g^{r-1} \arrow[r,"q_{(r-1)}"] & \cC_g^{r-2}.
	\end{tikzcd}
	\]
	These maps induce linear subspaces \(W_1 := \Im p_{(r)}^*\), \(W_2 := \Im p_{(r-1)}^*\), and \(W_{12} := W_1 \cap W_2\) of \(\tR^2(\cC_g^r)\). The forms $e_1^A$, $\nu$, $p_i^*e^A$, and $p_{ij}^* h$, (possibly) except for the form $p_{r-1,r}^* h$, all lie in $W_1$ or $W_2$. It follows from Theorem~\ref{thm:generate_deg_two} that  
	\[\tR^2(\cC_g^r) = (W_1 + W_2) + \R\cdot p_{r-1,r}^* h.\]
	
	Obviously the pullback of each form on \(\cC_g^{r-2}\) along the composition \(q_{(r-1)}\circ p_{(r-1)} = q_{(r-1)}\circ p_{(r)}\) is an element of \(W_{12}\).
	Conversely, we claim that each form in \(W_{12}\) is the pullback along this composition of some form on \(\cC_g^{r-2}\).
	Indeed, let \(\alpha\in W_{12}\) be any form; we may write \(\alpha = p_{(r)}^* \beta = p_{(r-1)}^*\gamma\) for forms \(\beta,\gamma\in \tR^2(\cC_g^{r-1})\). Let $\mu\in \tR^2(\cC_g)$ be the 2-form given by \(\mu=e^A/(2-2g)\); it follows that \(\int_{\cC_g/\cM_g} \mu = 1\), and by the base change formula we obtain
	\[\int_{p_{(r)}} p_r^*\mu = 1.\] 
	We then find by repeatedly using the base change formula and the projection formula:
	\begin{align*}	
		\beta &= \beta\wedge \int_{p_{(r)}} p_r^*\mu \\&=\int_{p_{(r)}} p_{(r)}^* \beta \wedge p_r^*\mu \\&= \int_{p_{(r)}} p_{(r-1)}^*\gamma\wedge p_r^*\mu \\&= \int_{p_{(r)}}p_{(r-1)}^*(\gamma\wedge p_{r-1}^*\mu) \\&= q_{(r-1)}^*\int_{q_{(r-1)}} \gamma\wedge p_{r-1}^*\mu,
	\end{align*}
	and therefore
	\[\alpha = p_{(r)}^*\beta = p_{(r)}^* q_{(r-1)}^* \int_{q_{(r-1)}} \gamma \wedge p_{r-1}^* \mu,\]
	which proves our claim.

	As pullbacks along tautological submersions are injective, we see that a basis of $ \tR^2(\cC_g^{r-1})$ pulls back to give bases of both $W_1$ and $W_2$, and that a basis of $ \tR^2(\cC_g^{r-2})$ pulls back to give a basis of $W_{12}$. Applying the induction hypothesis and using simple linear algebra we see that $W_1 + W_2$ has a basis consisting of exactly the forms mentioned in the Theorem, except for the form $p_{r-1,r}^* h$.
	
	If we can prove that \(p_{r-1,r}^* h \notin W_1+W_2\) then we may conclude that $\tR^2(\cC_g^r)$ has a basis as given in the Theorem. Suppose, therefore, that \(p_{r-1,r}^* h \in W_1 + W_2\); we can write \(p_{r-1,r}^* h = p_{(r)}^* \alpha + p_{(r-1)}^* \beta\) for some 2-forms \(\alpha,\beta\) on \(\cC_g^{r-1}\). As \(h\) is symmetric, we may even assume with no loss of generality that \(\alpha = \beta\):
	\[p_{r-1,r}^* h = p_{(r)}^* \alpha + p_{(r-1)}^*\alpha.\]
	Consider the map \[f: \cC_g^{r-1} \to \cC_g^r: (x_1,\dots,x_{r-1})\mapsto (x_1,\dots,x_{r-1},x_{r-1});\] this map is a section of both \(p_{(r)}\) and \(p_{(r-1)}\) and fits in a cartesian diagram
	\[
		\begin{tikzcd}
			\cC_g^{r-1} \arrow[r,"f"]\arrow[d,"p_{r-1}"] & \cC_g^r \arrow[d,"p_{r-1,r}"] \\
			\cC_g \arrow[r,"\Delta"] & \cC_g^2.
		\end{tikzcd}
	\]
	We then find:
	\[p_{r-1}^* e^A = p_{r-1}^*\Delta^* h = f^* p_{r-1,r}^* h = 2\alpha;\]
	so \(\alpha = \tfrac12 p_{r-1}^* e^A\), and
	\[p_{r-1,r}^*h = \tfrac12 p_{r-1}^* e^A + \tfrac12 p_r^*e^A \in \tR^2(\cC_g^r).\]
	Integration along the fibers of the morphism \(p_{(r)}: \cC_g^r \to \cC_g^{r-1}\) then yields:
	\[1 = \int_{p_{(r)}} p_{r-1,r}^* h = \int_{p_{(r)}} \tfrac12(p_{r-1}^* e^A + p_r^*e^A) = 0 + \tfrac12(2-2g), \]
	which contradicts with our assumption that \(g\geq 2\). We conclude that \(p_{r-1,r}^*h \notin W_1+W_2 \). 
		The theorem follows by induction.
\end{proof}

\begin{remark} We have observed in the above proof that for high values of $r$, no `new' tautological $2$-forms appear; more precisely, for $r>2$ the space $\tR^2(\cC_g^r)$ is spanned by pullbacks of $2$-forms in $\tR^2(\cC_g^2)$ along tautological submersions. This pattern generalizes to higher cohomological degrees: let \(d\geq 0\) be an integer, then for all \(r>2d\) the space \(\tR^{2d}(\cC_g^r)\) is spanned by pullbacks of tautological \(2d\)-forms on \(\cC_g^{2d}\) along tautological submersions \(\cC_g^r \to \cC_g^{2d}\). We refer to \cite[Section~4.9]{thesis} for more details. 
\end{remark}

\begin{proof}[Proof of Theorem~\ref{thm:ZK_inv}]
It follows from Lemma~\ref{prop:non-zero} that \( \frac{1}{\pi\i} \del\delbar\varphi  \) spans a subspace of  \(I^2(\cM_g)\) of dimension one.
	If \(g=2\) this concludes our proof, since \( \tR^2(\cM_g)\) is one-dimensional by Proposition~\ref{prop:dimR2Mg}. Assume now that \(g\geq 3\). Then by Lemma \ref{prop:some_not_exact} we have that \(I^2(\cM_g)\) is a proper subspace of \(\tR^2(\cM_g)\), and by Proposition~\ref{prop:dimR2Mg} the latter space is two-dimensional.  
	This gives the first part of the theorem. The second part of the theorem is then immediate from the first part and Lemma~\ref{prop:exist_unique_potential}. 
\end{proof}

\begin{remark} An argument very similar to the proof of Theorem~\ref{thm:degree_two} will show that the degree-two part $RH^2(\cC_g^r)$ of the tautological ring in cohomology has basis given by the classes $p_{ij}^*\Delta$ and $p_i^*K$ when $g=2$ and the classes  $p_{ij}^*\Delta$, $p_i^*K$ and $p^*\kappa_1$ when $g \geq 3$. This result is well known. Combined with Theorem~\ref{thm:degree_two} we obtain that for each $r \ge 0$ the canonical surjection $\tR^2(\cC_g^r) \to RH^2(\cC_g^r) \otimes \R$ has one-dimensional kernel. We conclude that in fact $(0) \neq I^2(\cC_g^r) = \R \cdot \frac{1}{\pi\i} \del\delbar\varphi $ for all $r \ge 0$.
\end{remark}

\bibliography{refs}
\bibliographystyle{plain}

\vspace{0.5cm}
\end{document}